\documentclass{article}
\usepackage[utf8]{inputenc}

\usepackage{amsmath}
\usepackage{amssymb}
\usepackage{amsfonts}
\usepackage{amsthm}
\usepackage{bbm}
\usepackage[mathscr]{eucal}
\usepackage{a4wide}
\usepackage{authblk}
\usepackage{float}
\usepackage{comment}
\usepackage{enumitem}

\usepackage[normalem]{ulem}

\usepackage{mathtools}
\mathtoolsset{showonlyrefs}

\usepackage{hyperref}
\hypersetup{
	colorlinks = true,
	urlcolor = {magenta},
	citecolor = {red},
	linkcolor = {blue}
}
\usepackage{doi}
\usepackage[square,numbers]{natbib}

\usepackage{xcolor}
\usepackage{todonotes}

\newcommand{\RNum}[1]{\uppercase\expandafter{\romannumeral #1\relax}}

\newcommand{\pr}{\mathbb{P}}								%	short hand for blackboard P

\newcommand{\Prob}[1]{\pr\left(#1\right)}					%	Standard probability command, 
\newcommand{\CProb}[2]{\pr\left(#1 \mid #2\right)}	%	Conditional probability command, 
\newcommand{\E}{\mathbb{E}}								%	short hand for blackboard E

\newcommand{\Exp}[1]{\E\left[#1\right]}					%	Standard expectation command, argument
\newcommand{\plim}{\ensuremath{\stackrel{\pr}{\rightarrow}}}	%	Convergence in probability
\newcommand{\dlim}{\ensuremath{\stackrel{d}{\rightarrow}}}		%	Convergence in distribution

\newcommand{\1}{\mathbbm{1}}								%	Indication shorthand command
\newcommand{\ind}[1]{\1_{\{#1\}}}	%	Command for indicator where argument is a condition
\newcommand{\indE}[1]{\1_{#1}}					%	Command for indicator where the argument is an event 

\newcommand\numberthis{\addtocounter{equation}{1}\tag{\theequation}}

\allowdisplaybreaks

\newtheorem{theorem}{Theorem}[section]
\newtheorem{lemma}[theorem]{Lemma}

\newtheorem{proposition}[theorem]{Proposition}

\newtheorem{conjecture}[theorem]{Conjecture}

\theoremstyle{definition}

\newtheorem{assumption}[theorem]{Assumption}
\newtheorem{properties}[theorem]{Properties}
\newtheorem{remark}[theorem]{Remark}

\numberwithin{equation}{section}

\newcommand{\cA}{\mathcal A}
\newcommand{\cB}{\mathcal B}

\newcommand{\cC}{\mathcal C}
\newcommand{\cD}{\mathcal D}

\newcommand{\cG}{\mathcal G}
\newcommand{\cI}{\mathcal I}

\newcommand{\cL}{\mathcal L}
\newcommand{\cM}{\mathcal M}
\newcommand{\cN}{\mathcal N}
\newcommand{\cP}{\mathcal P}
\newcommand{\cR}{\mathcal R}

\newcommand{\cS}{\mathcal S}

\newcommand{\cZ}{\mathcal Z}

\newcommand{\Z}{\mathbb Z}
\newcommand{\N}{\mathbb N}
\newcommand{\R}{\mathbb R}

\newcommand{\sC}{\mathscr{C}}

\newcommand{\bD}{\boldsymbol{D}}

\newcommand*{\be}{\begin{equation}}
	\newcommand*{\ee}{\end{equation}}
\newcommand*{\ba}{\begin{aligned}}
	\newcommand*{\ea}{\end{aligned}}

\newcommand{\eps}{\epsilon}

\newcommand{\invisible}[1]{}

\title{On random bipartite graphs evolving by degrees}

\begin{document}
	\author{Neeladri Maitra\thanks{e-mail: nmaitra@illinois.edu}}\affil{Department of Mathematics\\
University of Illinois Urbana-Champaign}
    \date{}
	
	\maketitle

\begin{abstract}
         In this paper, we study a bipartite analogue of the `random graphs evolving by degrees' process. We are given a bipartitioned set of vertices $V$ into two disjoint parts $\cL$ and $\cR$ and possibly unequal positive constants $\alpha$ and $\beta$. The graph evolves starting from $B_0$, the empty graph (with only isolated vertices). Given $B_t$, a non-adjacent vertex pair $u \in \cL, v \in \cR$ is sampled with probability proportional to $(d_u(t)+\alpha)(d_v(t)+\beta)$, and the edge $\{u,v\}$ is included to $B_t$ to form $B_{t+1}$, where $d_u(t)$ is the degree of $u$ in $B_t$. For this model, we establish the threshold for the appearance of a giant component, the connectivity threshold for the associated multigraph variant, and provide a superlinear lower bound on the connectivity threshold for the simple graph case. 
         
         For the proof of the giant component result, our methods involve setting up an exact coupling of the multigraph case with a bipartite configuration model and using existing results on the giant of bipartite configuration models. This is an adaptation of the technique of Janson and Warnke \cite{janson2021preferential} where they treat the unipartite case similarly. For the connectivity results, we first set up a formula for the exact probability of the occurrence of certain connectivity events in the multigraph process, which is interesting in its own right. To then derive the connectedness threshold, we analyze a particular case of it à la Pittel \cite{pittel2010random}. For the superlinear connectivity lower bound in the simple graph case, we establish and use a tail bound on the number of isolated vertices in the multigraph process, together with a change of measure statement to go from the multi to the simple graph process.

	\end{abstract}

%    \tableofcontents

    \noindent  \smallskip
        %{\bf Keywords:}  First-passage percolation, long-range first-passage percolation, competition, competing first-passage percolation, continuous-time branching processes.
        
        %{\bf MSC Subject Classifications:} Primary: 60K35, Secondary: 60C05.

	\section{Introduction}

This paper is about a random graph process, dynamically evolving over time. Such models are motivated by different perspectives. For example, we observe most real-world networks grow over time, imagine social media networks of friendships, or research article networks, where new articles come into the literature, and one citing the other represent connections between them. Random graph models designed to mimic growing real-world networks include the Barabási-Albert preferential attachment model \cite{barabasi1999emergence} and related variants; we point the reader to \cite{barabasi2013network, van2024random} for detailed discussions on the topic. However, sometimes it is also worth looking at variants of existing random graph evolutions, which shed light on the formation of certain structures in them. Consider the work \cite{bohman2001avoiding} for example, where the authors study a variant of the classical Erdős-Rényi random graph process, with an attempt to \emph{delay} the formation of a giant component. We also refer the reader to the works \cite{molloy2022degree, riordan2011explosive, riordan2012achlioptas, bohman2001avoiding, pittel2010random, janson2021preferential} where variants of the Erdős-Rényi process are studied.

In a by-now celebrated paper \cite{pittel2010random}, Pittel initiates the study of an $\alpha$-variant of the Erdős-Rényi random graph evolution process, $\alpha>0$ a constant. In this process, one starts with an empty graph (with only isolated vertices) $G_0$ on $n$ vertices, and given the graph $G_t$, a currently non-adjacent vertex pair $\{u,v\}$ is sampled with a probability proportional to $(d_u(t)+\alpha)(d_v(t)+\alpha)$, where $\alpha>0$ is a fixed real, and the edge $\{u,v\}$ is added to $G_t$ to obtain $G_{t+1}$ from it. Let us informally call this the \emph{$\alpha$-dynamics}. Observe that this dynamics can be interpreted as a generalization of the classical Erdős-Rényi (E-R) dynamics, where the sampling is instead done uniformly, in that, letting $\alpha \to \infty$ at least intuitively recovers the latter. Indeed, as Pittel \cite{pittel2010random} proves, many known features of the E-R dynamics are recovered from the analogous $\alpha$ expressions, by letting $\alpha \to \infty$. 

There are certain qualitative features of the E-R dynamics that persist when considering bipartite, or in general multipartite variants. For example, in the bipartite case with equal sized partitions, say of size $n$ each, a giant component emerges at a linear scale, namely at around $t=n$ edges (e.g., by \cite[Exercise 2.4.7]{frieze2015introduction}), which is of the same order as of the giant component emergence threshold for the original unipartite E-R graph, around $n/2$ edges \cite{erd6s1960evolution}. Note that even though the thresholds are both linear, they differ by a constant factor $1/2$, a consequence of the bipartite structure imposed. Not just the emergence of the giant component, even the connectivity threshold for both these processes appear at the same scale (see \cite[Chapter 4]{frieze2015introduction}), around $(n \ln n)/2$ edges for the unipartite case, while around $n \ln n$ edges for the bipartite case; the factor $1/2$ difference also persisting here.

For the unipartite case, it is believed that when the number of edges included so far satisfies $t=\Theta(n)$, the final graph of the $\alpha$-dynamics, in many ways, belong in the same universality class as the final graph of the E-R dynamics. Pittel \cite{pittel2010random} proves for example, a giant component emerges in the former at around $t=f(\alpha)n$ many edges, with $f(\alpha)=\frac{1}{2(1+\alpha^{-1})} \to 1/2$ as $\alpha \to \infty$, recovering E-R behavior, and in fact the critical window for the appearance of a giant component has width of order $n^{4/3}$ about $f(\alpha)n$ just like the E-R dynamics, see \cite{pittel2010random, janson2021preferential}. In contrast, when the number of edges is superlinear, the effect of $\alpha$ should be strong enough to cause a change in behavior from the E-R case. The first instance of a behavior exhibiting this was proved by Pittel, who showed in \cite[Theorem 2]{pittel2010random} that the connectivity threshold for the multigraph version of the $\alpha$-dynamics occur at about $n^{1+\alpha^{-1}}$ many edges, much larger than $(n \ln n)/2$ for the E-R dynamics. He conjectures that this should also be the connectivity threshold for the simple graph $\alpha$-dynamics case, something that still stands widely open today. Towards this, we take a small first step by showing that the connectivity threshold for the simple graph case is \emph{strictly superlinear}, in that there is an explicit $\delta>0$, such that the simple graph is disconnected with probability tending to $1$ after the inclusion of at most $n^{1+\delta}$ edges, see Theorem \ref{thm:sg_LB} below and the discussion in Section \ref{sec:sg_uni_lb}.

\paragraph{Model definition, and some motivating questions.} In this paper, we are motivated by the similarity in behavior between the uni and the bipartite case for the E-R dynamics, as discussed above. We study a natural bipartite analogue of the $\alpha$-dynamics. We are given  a set of vertices $V$ and a partition of it into two disjoint parts $\cL$ and $\cR$. One starts with the empty graph $B_0$ (with only isolated vertices), and given $B_t$, one constructs $B_{t+1}$ from it by randomly sampling a pair $\{u,v\}$ of currently non-adjacent vertices with $u \in \cL$, $v \in \cR$, and including the edge $\{u,v\}$ in $B_t$. This pair is sampled with a probability proportional to $$(d_u(t)+\alpha)(d_v(t)+\beta),$$ where $\alpha, \beta$ are fixed positive reals (possibly unequal). Let us call this the $(\alpha,\beta)$-dynamics. We have been driven by an attempt to answer the following questions.
\begin{itemize}
    \item Does the bipartite $(\alpha,\beta)$-dynamics for $\alpha=\beta$ behave similarly to the bipartite E-R dynamics, in that it exhibits some similarity with the unipartite analogue?
    \item What sort of difference does taking $\alpha \neq \beta$ make?
    \item Can we prove a result for the \emph{simple} graph $(\alpha,\beta)$-dynamics that shows that when the number of edges is superlinear, already the model is quite different from the E-R case with the same number of edges? As remarked before, for the multigraph unipartite version, \cite[Theorem 2]{pittel2010random} establishes this.
    \item Are there constant factor gaps between the thresholds in the $(\alpha,\beta)$-dynamics and its unipartite analogue, like in the E-R case? 
\end{itemize}

We answer some of these questions partially in this paper. To formally state our theorems, we need to set up some notations, which let us get into next.

\paragraph{Some notations.} We use the standard order notations $O(\cdot), \Omega(\cdot),o(\cdot),\omega(\cdot),\Theta(\cdot)$ in their usual sense. For non-negative real seqeuences $(x_n)_{n \geq 1}, (y_n)_{n \geq 1}$, we sometimes write $x_n \ll y_n$ which is the same as $x_n=o(y_n)$, and $x_n \gg y_n$, which is the same as $x_n=\omega(y_n)$. We write \textbf{whp} for `with probability tending to $1$'. By $\mathrm{Poi}(\lambda)$, we denote a Poisson random variable with mean $\lambda$. By $\mathrm{Bin}(n,p)$, we denote a binomial random variable with $n$ trials and success probability $p$ of each trial. By $\mathrm{Exp}(\lambda)$ we define an exponential random variable with rate $\lambda$, i.e., with density function $\lambda e^{-\lambda x}$, $x \geq 0$. For a set $E$, by $|E|$ we simply denote its size.

\subsection{Main results}

\paragraph{Key assumptions.} We work with the following key assumptions throughout the paper.
\begin{itemize}
    \item Denoting $L_n=|\cL|$ and $R_n=|\cR|$, where $L_n,R_n$ are sequences diverging to infinity, we assume that the parts $\cL$ and $\cR$ are \emph{comparable} in size, i.e., \begin{align*}
    \lim_{n \to \infty} \frac{R_n}{L_n} = \gamma,\;\;\text{for some}\;\;\gamma \in \R_+=(0,\infty). \numberthis \label{eq:assump_both_pos_prop}
\end{align*}
\item Throughout the note we assume $\alpha,\beta>0$.
\end{itemize}

The following definitions are useful to phrase some of our results concisely. Define maps $\rho:\{\cL,\cR\}\to \{\alpha,\beta\}$ and $\zeta:\{\cL,\cR\}\to \{1+\gamma,1+1/\gamma\}$ by
\begin{align*}
\rho(\cL)=\alpha,\rho(\cR)=\beta,\;\;\text{and}\;\;\zeta(\cL)=1+\gamma,\zeta(\cR)=1+1/\gamma.\numberthis \label{eq:maps_rho_zeta}
\end{align*}
When $\alpha \neq \beta$, $\rho$ is invertible, and we denote by $\rho^{-1}$ its inverse map. Similarly when $\gamma\neq 1$, we denote the inverse map of $\zeta$ by $\zeta^{-1}$. For any $\eps>0$, define constants
\begin{align*}
    p_{L}(\eps)=\frac{\sqrt{\gamma^{-1}(1+\frac{1}{\alpha})(1+\frac{1}{\beta})}}{\sqrt{\gamma^{-1}(1+\frac{1}{\alpha})(1+\frac{1}{\beta})}+\frac{\alpha}{1+\eps}},\;\; p_{R}(\eps)=\frac{\sqrt{\gamma(1+\frac{1}{\alpha})(1+\frac{1}{\beta})}}{\sqrt{\gamma (1+\frac{1}{\alpha})(1+\frac{1}{\beta})}+\frac{\beta}{1+\eps}}. \numberthis \label{eq:supercrit_succ_probs}
\end{align*}

\subsubsection{Giant component threshold}

\begin{theorem}[Emergence of a giant component]\label{thm:giant}
  Consider the bipartite random graph process $(B_t)_{t \geq 0}$. Assume \eqref{eq:assump_both_pos_prop}. Define \begin{align*}
          t_c:=\frac{\sqrt{\gamma}}{(\gamma+1)\sqrt{(1+1/\alpha)(1+1/\beta)}}.\numberthis \label{eq:t_c_def}
      \end{align*}
For any $t\geq 0$, let $\cC_1(t)$ and $\cC_2(t)$ respectively denote the largest and the second largest connected components of the graph $B_t$.
  \begin{itemize}
      \item \textbf{Supercritical regime.} If $t=t_{\eps}^+=t_c(1+\eps)(L_n+R_n)$ for some constant $\eps>0$, then
\begin{align*}
    \frac{|\cC_1(t_{\eps}^+)|}{L_n+R_n}\plim \frac{\xi_L(\eps)+\gamma \xi_R(\eps)}{1+\gamma}>0. \numberthis \label{eq:lim_giant}
\end{align*}
To describe the limit, recall \eqref{eq:supercrit_succ_probs}, and let $\eta_L=\eta_L(\eps)$ be the smallest solution to the equation
\begin{align*}
    p_R(\eps)\left(\frac{1-p_L(\eps)}{1-\eta_Lp_L(\eps)} \right)^{\alpha+1}=1-(1-p_R(\eps))(1-\eta_L)^{-(\beta+1)^{-1}}. \numberthis \label{eq:def_eta_L}
\end{align*}
Then, writing simply $\xi_L$ and $\xi_R$ for respectively $\xi_L(\eps)$ and $\xi_R(\eps)$,
\begin{align*}
    \xi_L=1-\left(\frac{1-p_L(\eps)}{1-\eta_Lp_L(\eps)} \right)^{\alpha};\;\;\text{and}\;\; \xi_R=1-\left(\frac{1-p_R(\eps)}{1-\eta_Rp_R(\eps)} \right)^{\beta}\;\;\text{with}\;\;\eta_R=\left(\frac{1-p_L(\eps)}{1-p_L(\eps)\eta_L} \right)^{\alpha+1}.
\end{align*}

Furthermore, the second largest component $\cC_2(t_{\eps}^+)$ satisfies $\frac{|\cC_2(t_{\eps}^+)|}{L_n+R_n}\plim 0.$
\item \textbf{Subcritical regime.} If $t=t_{\eps}^-=t_c(1-\eps)(L_n+R_n)$ for some $\eps>0$, then $\frac{|\cC_1(t_{\eps}^-)|}{L_n+R_n}\plim 0.$  \end{itemize}
\end{theorem}

\subsubsection{Connectivity properties}
Let us introduce a closely related multigraph version $(B^*_t)_{t \geq 0}$ of the process $(B_t)_{t \geq 0}$. At step $t=0$ we begin from the empty graph $B^*_0$ as usual, and at each step, to construct the graph $B^*_{t+1}$ from $B^*_t$, we sample an edge $\{u,v\}$ with $u \in \cL$ and $v \in \cR$ with probability
\begin{align*}
    \frac{(d_u(t)+\alpha)(d_v(t)+\beta)}{\sum_{x \in \cL,y \in \cR}(d_x(t)+\alpha)(d_y(t)+\beta)}=\Big(\frac{(d_u(t)+\alpha)}{\sum_{x \in \cL}(d_x(t)+\alpha)}\Big)\Big(\frac{(d_v(t)+\beta)}{\sum_{y \in \cR}(d_y(t)+\beta)}\Big). \numberthis \label{eq:multi_prob}
\end{align*} 

In contrast, recall that given $B_t$, to construct $B_{t+1}$, we add a (yet non-existing) edge between $u \in \cL$, $v \in \cR$ with probability
\begin{align*}
    \frac{(d_u(t)+\alpha)(d_v(t)+\beta)}{\sum_{x \in \cL, y \in \cR, \{x,y\}\notin E(\cG_t)}(d_x(t)+\alpha)(d_y(t)+\beta)}. \numberthis \label{eq:actual_prob}
\end{align*}
Thus, in the multigraph process $(B^*_t)_{t \geq 0}$, an edge that has been sampled already before is allowed to be sampled again, potentially giving rise to multiple edges between a pair of vertices, while in the simple process $(B_t)_{t \geq 0}$, such with replacement samplings are barred.
Consider the bipartite multigraph process $(B^*_t)_{t \geq 0}$. Recall the maps $\rho$ and $\zeta$ from \eqref{eq:maps_rho_zeta}.
\begin{theorem}[Multigraph connectedness threshold]\label{thm:multi_conn}
    Consider the bipartite multigraph process $(B^*_t)_{t \geq 0}$.
    Define
    \begin{align*}
        \tau_{c,n}:=(L_n+R_n)^{1+(\alpha\wedge \beta)^{-1}}.\numberthis \label{eq:conn_threhsold}
    \end{align*}
Assume $t=t_n$ is a sequence of $n$. Then,
    \begin{align*}
        \lim_{n \to \infty}\Prob{B^*_t\;\;\text{is connected}}=\begin{cases}
            &0,\hspace{75 pt}\;\;\text{if}\;\;t\ll \tau_{c,n},\\& 1,\hspace{75 pt}\;\;\text{if}\;\;t\gg \tau_{c,n}.
        \end{cases}
    \end{align*}
Furthermore, if $\lim_{n \to \infty}\frac{t}{\tau_{c,n}}=x \in (0, \infty)$,
\begin{align*}
        \lim_{n \to \infty}\Prob{B^*_t\;\;\text{is connected}}=\begin{cases}
            &\exp \left(-{(\frac{\alpha \wedge \beta}{x})^{\alpha \wedge \beta}}{[\zeta(\rho^{-1}(\alpha \wedge \beta))]^{-1-\alpha 
            \wedge\beta}} \right),\hspace{20 pt}\;\;\text{if}\;\;\alpha \neq \beta,\\& \exp\left(-\left(\frac{\alpha}{x}\right)^{\alpha}\left((1+\gamma)^{-1-\alpha}+(1+1/\gamma)^{-1-\alpha} \right) \right),\;\;\text{if}\;\;\alpha=\beta.
        \end{cases}\\\numberthis \label{eq:connectivity_limits}
    \end{align*} 

\end{theorem}
\begin{remark}[Effect of partition sizes]
Consider the situation when $\gamma$ tends to either $0$ or $\infty$, thus one partition is much larger in size than the other one. In this case recalling \eqref{eq:t_c_def}, $t_c \to 0$, so that a giant component emerges after the inclusion of $\eps(L_n+R_n)$ many edges for any $\eps>0$. On the other hand, note that the connectivity threshold \eqref{eq:conn_threhsold} is independent of $\gamma$. However, looking at \eqref{eq:connectivity_limits}, for $\alpha \neq \beta$, if $\rho^{-1}(\alpha \wedge \beta)$ is the \emph{smaller} partition, so that $\zeta(\rho^{-1}(\alpha \wedge \beta))\to \infty$, we see that the limit approaches $1$, which suggests in this case connectivity is \emph{easier} to achieve when $t$ is of the order of $\tau_{c,n}$.
\end{remark}
Following Pittel \cite{pittel2010random}, we leave the following conjecture.
\begin{conjecture}[Simple graph connectedness threshold]\label{conj:simple_con}
    Recall $\tau_{c,n}$ from \eqref{eq:conn_threhsold}. When $\alpha \wedge \beta > 1$, the simple graph $\cG_t$ is \textbf{whp} connected for $t \gg \tau_{c,n}$, while it is \textbf{whp} disconnected for $t \ll \tau_{c,n}$.
\end{conjecture}
As the simple graph can ever have at most $L_n R_n=\Theta(L_n^2)$ many edges, it is interesting to think about the location of its connectivity threshold when $\alpha \wedge \beta \leq 1$. It is suggestive that the simple graph stays disconnected till the inclusion of the \emph{very last few} edges, when it essentially becomes a complete bipartite graph, see also the discussion by Pittel around the last paragraph of \cite[pg-621]{pittel2010random}.  %Perhaps in this case the multigraph process needs (in order) more that $L_n R_n$ many edges to be connected, while the simple graph stays disconnected until the inclusion of the very \emph{last few} edges? Can one quantify what \emph{last few} is in this context? We don't have enough intuition to comment anything about these questions at this point.
Note that the last result shows that the behavior of the $(\alpha,\beta)$-dynamics in the multigraph setting, when the number of edges is superlinear (in the total number of vertices) can be markedly different from the E-R dynamics with the same number of edges, as in the latter case connectivity happens around order $L_n \ln L_n$ many edges, both for the simple and multigraph case. Conjecture \ref{conj:simple_con} further suggests that the simple graph $(\alpha,\beta)$-dynamics should also exhibit behavior analogous to the multigraph counterpart. Towards this, we ask the following easier question: can we prove \emph{some} result for the simple graph process that shows that the connectivity behavior of it is quite different from the E-R case? Indeed, as the next result shows, this is possible.
\begin{theorem}[Simple graph connectedness lower bound]\label{thm:sg_LB}
    Consider the bipartite simple graph process $(B_t)_{t \geq 0}$. Define
    \begin{align*}
        \cZ(\alpha,\beta):=\left(\frac{1}{2(1+\alpha)}\wedge\frac{1}{2(1+\beta)}\right)\wedge\left( \left(\frac{1}{4+\alpha}\wedge \frac{1}{\beta}  \right)\vee\left( \frac{1}{4+\beta}\wedge \frac{1}{\alpha}\right)\right).\numberthis\label{eq:def_cZ}
    \end{align*}
For any $\delta \in \left(0,\cZ(\alpha,\beta)\right)$, the graph $B_{t(\delta)}$ is disconnected \textbf{whp}, where $t(\delta)=(L_n+R_n)^{1+\delta}=\Theta(L_n^{1+\delta})$.
\end{theorem}
As a consequence of the last theorem, we note that the threshold for connectivity of the simple graph process $(B_t)_{t \geq 0}$ is thus at least (in order) $L_n^{1+\cZ(\alpha,\beta)}$ many edges, much larger in order than $L_n \ln L_n$, the corresponding connectivity threshold for the simple random E-R bipartite graph process.

\paragraph{Organization of the rest of the paper.} In Section \ref{sec:giant_proof} we discuss the strategy and give the proof of the giant component result Theorem \ref{thm:giant}. In Section \ref{sec:connectivity properties}, we provide the necessary required results and the proofs of Theorems \ref{thm:multi_conn} and \ref{thm:sg_LB}. We close with a small discussion in Section \ref{sec:disc}.

\paragraph{Acknowledgements.} Thanks to Partha S. Dey for discussions and encouragement. Thanks to Partha S. Dey and Remco van der Hofstad for reading a first draft, and giving me many helpful suggestions that significantly improved the article.

\section{Strategy and proof of emergence of giant}\label{sec:giant_proof}
This proof is an adaptation of techniques developed by Janson and Warnke \cite{janson2021preferential} for the unipartite setting.

\paragraph{Approximating with the multigraph process.}
The first step is to approximate the simple graph process $(B_t)_{t \geq 0}$ with the multigraph process $(B^*_t)_{t \geq 0}$.
\begin{proposition}\label{prop:multi_approx}
    Assume $t=t_n\leq C(L_n+R_n)$ for some constant $C>0$. Further assume $\alpha,\beta>0$. Consider a deterministic sequence of bipartite multigraphs $(H_i)_{0 \leq i \leq t}$ on the vertex set $V=\cL \cup \cR$, with the property that $H_{i+1}$ is obtained from $H_i$ by including a single edge. Then there exists a constant $K=K(C,\alpha,\beta)>0$ such that as $n \to \infty$,
    \begin{align*}
        \Prob{B_i=H_i\;\;\forall\;\;1\leq i \leq t}\leq K\cdot\Prob{B^*_i=H_i\;\;\forall\;\;1\leq i \leq t}+o(1). 
    \end{align*}
\end{proposition}
Thus, when the number of edges is at most linear, rare events for the multigraph process are also rare events for the simple graph process. This result is useful, if we have some control over the multigraph process. Indeed, as we will prove, it is a bipartite configuration model.

\paragraph{Bipartite configuration models.} 
Consider a bipartition $V=\cL \cup \cR$ of a vertex set $V$. Assume that we are given non-negative integer sequences $\bD^{(\cL)}=(d_u)_{u \in \cL}$ and $\bD^{(\cR)}=(d_v)_{v \in \cR}$, satisfying
\begin{align*}
    \sum_{u \in \cL}d_u=\sum_{v \in \cR}d_v. \numberthis \label{eq:left_eq_right}
\end{align*}
We form a bipartite multigraph with $\bD=(\bD_{\cL},\bD_{\cR})$ as its bi-degree sequence as follows. Associate to each vertex $u \in V$ a set of $d_u$ many half-edges $H_u:=\{l^{(u)}_1,\dots,l^{(u)}_{d_u}\}$. Take a uniformly random pairing $\cP$ of the \emph{left} half-edges $\bigcup_{u \in \cL}H_u$ with the \emph{right} half-edges $\bigcup_{v \in \cR}H_v$. Given the pairing, to construct a bipartite multigraph on $V$, let $u$ and $v$ share $n_{\cP}(u,v)$ many edges between them where $n_{\cP}(u,v)$ is the number of elements of the form $(l^{(u)}_i,r^{(v)}_j)$ in the pairing $\cP$. We call the obtained bipartite multigraph the \emph{bipartite configuration model} corresponding to vertex partitions $\cL$ and $\cR$, and bi-degree sequence $\bD=(\bD_{\cL},\bD_{\cR})$, and denote it as $\mathrm{BCM}(\bD_{\cL},\bD_{\cR})$. The key result for the multigraph process $(B^*_t)_{t \geq 0}$ is the following.
\begin{proposition}\label{prop:BCM_approx}
    For any $t\geq 0$, and any bi-degree sequence $(\bD_{\cL},\bD_{\cR})$, the bipartite multigraph $B^*_t$ conditioned on having the bi-degree sequence $(\bD_{\cL},\bD_{\cR})$, has the same distribution as a $\mathrm{BCM}(\bD_{\cL},\bD_{\cR})$.
\end{proposition}
With the last result at hand, to conclude Theorem \ref{thm:giant}, our task reduces to analyzing the giant component of bipartite configuration models (BCMs), and obtaining some control over the behavior of the bi-degree sequence of $(B^*_t)_{t \geq 0}$, when $t$ is at most linear.
\paragraph{The giant in bipartite configuration models.}
The giant in bipartite configuration models was analyzed in \cite{van2022phase}. We discuss one of their main results in this section. Consider a bi-degree sequence $\bD=(\bD_{\cL},\bD_{\cR})$ satisfying \eqref{eq:left_eq_right}, and a $\mathrm{BCM}(\bD_{\cL},\bD_{\cR})$ on the bipartitioned vertex set $V=\cL \cup \cR$. Let $U_{\cL}$ and $U_{\cR}$ be respectively uniformly sampled elements from $\cL$ and $\cR$, and let $|\cL|=L_n$ and $|\cR|=R_n$ be both diverging sequences of $n$.
\begin{assumption}\cite[Assumption 2.3]{van2022phase}.\label{assump:giant_BCM} Let $p_n(k)=\Prob{d_{U_{\cL}}=k}$ and $q_n(k)=\Prob{d_{U_{\cR}}=k}$. 
    \begin{itemize}
        \item There exists limiting random variables $D_{\cL}, D_{\cR}$ with respective probability mass functions $p=(p_0,p_1,\dots)$ and $q=(q_0,q_1,\dots)$ such that $p_n(k) \to p_k$ and $q_n(k) \to q_k$ for each $k \geq 0$. In other words, $d_{U_{\cL}} \dlim D_{\cL}\;\;\text{and}\;\;d_{U_{\cR}} \dlim D_{\cR}.$
        \item Furthermore, $\Exp{D_{\cL}},\Exp{D_{\cR}}< \infty$, and $\Exp{d_{U_{\cL}}}\to \Exp{D_{\cL}}\;\;\text{and}\;\;\Exp{d_{U_{\cR}}}\to \Exp{D_{\cR}}.$
    \end{itemize}
\end{assumption}

To state the next result, we need the following useful definitions. For a random variable $X$, denote for $z \in [0,1]$ its probability generating function evaluated at $z$ as
\begin{align*}
    G_X(z):=\Exp{z^X}.\numberthis \label{eq:gen_func}
\end{align*}
For a $\Z_{>0}:=\{1,2,3,\dots\}$ valued random variable $X$, let us introduce the \emph{shifted size-biased} random variable $\underline{X}$ which has mass function given by,
\begin{align*}
\Prob{\underline{X}=k}=\frac{(k+1)\Prob{X=k+1}}{\Exp{X}},\;\;\text{for}\;\;k\in \Z_{\geq 0}:=\{0,1,2,\dots\}. \numberthis \label{eq:shifted_size_biased}
\end{align*}

We can now state the main result on the largest component of BCMs. Denote by $\cC_1$ and $\cC_2$ respectively the first and second largest components of a $\mathrm{BCM}(D_{\cL},\bD_{\cR})$. 

\begin{theorem}\label{thm:BCM_PT_vdH}\emph{\cite[Theorem 2.11, Corollary 2.12]{van2022phase}.} Let Assumption \ref{assump:giant_BCM} hold. Further assume $p_2+q_2<2$. 
\begin{itemize}
    \item \textbf{Supercritical regime.} Assume further the following \emph{supercriticality} condition holds,
    \begin{align*}
    \Exp{\underline{D_{\cL}}}\Exp{\underline{D_{\cR}}}>1. \numberthis \label{eq:BCM_supercrit}
    \end{align*}
    Let $\xi_L:=1-G_{D_{\cL}}(\eta_L)$, with $\eta_L$ being the smallest solution of the fixed point equation $\eta_L=G_{\underline{D_{\cR}}}(G_{\underline{D_{\cL}}}(\eta_L)).$ Further, let $\xi_R=1-G_{D_{\cR}}(\eta_R)$, with $\eta_R:=G_{\underline{D_{\cL}}}(\eta_L)$. Then,
    \begin{align*}
        \frac{|\cC_1|}{L_n+R_n}\plim \frac{\xi_L+\gamma \xi_R}{1+\gamma}>0\;\;\text{and}\;\;\frac{|\cC_2|}{L_n+R_n}\plim 0,\;\;\text{where}\;\;\gamma:=\frac{\Exp{D_{\cL}}}{\Exp{D_{\cR}}}.\;\;\numberthis \label{eq:BCM_giant_sup_lim}
    \end{align*}
    \item \textbf{Subcritical regime.} If the supercriticality condition \eqref{eq:BCM_supercrit} does not hold, then $\frac{|\cC_1|}{L_n+R_n}\plim 0.$
\end{itemize}
\end{theorem}
We remark that the limit $\frac{\xi_L+\gamma\xi_R}{1+\gamma}$ of the scaled size of the giant component appearing in \eqref{eq:BCM_giant_sup_lim} is strictly positive, and that is part of the assertion of the theorem. The similarity of the statement of the last theorem with the statement of Theorem \ref{thm:giant} should be noted. We next discuss the behavior of the bi-degree sequence of the random multigraph $B^*_t$.

\paragraph{Negative binomial bi-degree sequence of $B^*_t$.} Let us begin by recalling the negative binomial distribution, and some of its properties. The \emph{negative binomial distribution} with shape parameter $\alpha>0$ and success probability $p$, has probability mass function
\begin{align*}
    f(k)=\frac{\prod_{0 \leq j \leq k}(\alpha+j)}{k!}(1-p)^{\alpha}p^k,\;\;\text{for}\;\;k \in \Z_{\geq 0}=\{0,1,2,\dots\}.\numberthis \label{eq:neg_bin_pmf}
\end{align*}

The following lemma is \cite[Lemma 3.1]{janson2021preferential} and we don't include a proof. We have added the third statement, it easily follows from \eqref{eq:neg_bin_pmf}. Recall the \emph{shifted size-biased} distribution from \eqref{eq:shifted_size_biased}. For $x \in \R$ and an integer $k \geq 0$, define $\langle x \rangle_k=\prod_{j=0}^{k-1}(x-j)$, with the convention that $\langle x \rangle_0=1$.
\begin{lemma}[Properties of negative binomial]\label{lem:nb_properties}
    Let $Y \sim \mathrm{NB}(\alpha,p)$ with $\alpha \in (0,\infty)$ and $p \in (0,1)$.
    \begin{itemize}
        \item [a.] $\Exp{Y}=\frac{\alpha p}{1-p}$, and for $k \geq 0$, $\Exp{\langle Y \rangle_k}=(\Exp{Y})^k\prod_{j=1}^k(1+j/\alpha)$.
        \item [b.] $G_Y(z)=\left(\frac{1-p}{1-pz} \right)^{\alpha}$.
        \item [c.] $\underline{Y}\sim \mathrm{NB}(\alpha+1,p)$.
    \end{itemize}
\end{lemma}

Let $U_{\cL}$ and $U_{\cR}$ respectively be uniformly random elements from $\cL$ and $\cR$, and $p_t(k):=\Prob{d_{U_{\cL}}(t)=k}$, $q_t(k):=\Prob{d_{U_{\cR}}(t)=k}$, where for any $u \in V$, $d_u(t)$ denotes its degree in $B^*_t$. 

\begin{proposition}\label{prop:neg_bin_left_right_deg_seq}
    Assume $t=\Theta(L_n+R_n)$ and $\alpha,\beta>0$. Consider random variables $D_{n,\cL}\sim \mathrm{NB}\left(\alpha,t/(t+\alpha L_n) \right), D_{n,\cR}\sim\mathrm{NB}\left(\beta,t/(t+\beta R_n)\right).$ Then for any $k \geq 0$,
    \begin{align*}
        p_t(k)=\Prob{D_{n,\cL}=k}+o(1),\;\;q_t(k)=\Prob{D_{n,\cR}=k}+o(1).
    \end{align*}
    Furthermore, for any $k \geq 1$ and $\square \in \{\cL,\cR\}$, $\Exp{\left(d_{U_\square}(t) \right)^k}=\Exp{D_{n,\square}^k}+o(1)$.
\end{proposition}
All these results at hand, we can straightaway provide the proof of Theorem \ref{thm:giant}.

\begin{proof}[Proof of Theorem \ref{thm:giant}, subject to Theorem \ref{thm:BCM_PT_vdH}, and Propositions \ref{prop:multi_approx}, \ref{prop:BCM_approx} and \ref{prop:neg_bin_left_right_deg_seq}.]
    
    For any $\eps>0$ constant, define random variables $D_{\cL}(\eps)\sim \mathrm{NB}\left(\alpha,p_L(\eps) \right),\;\; D_{\cR}(\eps)\sim \mathrm{NB}\left(\beta,p_R(\eps) \right)$, where recall $p_L(\eps)$ and $p_R(\eps)$ from \eqref{eq:supercrit_succ_probs}.
%\begin{align*}
    %p_{L}(\eps)=\frac{\sqrt{\gamma^{-1}(1+\frac{1}{\alpha})(1+\frac{1}{\beta})}}{\sqrt{\gamma^{-1}(1+\frac{1}{\alpha})(1+\frac{1}{\beta})}+\frac{\alpha}{1+\eps}},\;\; p_{R}(\eps)=\frac{\sqrt{\gamma(1+\frac{1}{\alpha})(1+\frac{1}{\beta})}}{\sqrt{\gamma (1+\frac{1}{\alpha})(1+\frac{1}{\beta})}+\frac{\beta}{1+\eps}}. \numberthis \label{eq:supercrit_succ_probs}
%\end{align*}
    Let us only prove the supercritical statement of Theorem \ref{thm:giant}, the proof of the subcritical statement is completely analogous. To this end, reall $t_c$ from \eqref{eq:t_c_def}, and let $t=(1+\eps)t_c(L_n+R_n)$ for some $\eps>0$. 

    For this value of $t$, thanks to Proposition \ref{prop:neg_bin_left_right_deg_seq}, we have $D_{n,\cL} \dlim D_{\cL}(\eps)$ and $D_{n,\cR}\dlim D_{\cR}(\eps)$ with convergence of the respective first moments. For any $\delta>0$, let us define the event 
    \begin{align*}
        \cA_{\delta}:=\left\{\left|\frac{|\cC_1(t^+_{\eps})|}{L_n+R_n}-\frac{\xi_L+\gamma \xi_R}{1+\gamma} \right|>\delta \right\},
    \end{align*}
    where $\gamma=\frac{\Exp{D_{\cL}(\eps)}}{\Exp{D_{\cR}(\eps)}}$, and where $\xi_L:=1-G_{D_{\cL}(\eps)}(\eta_L)$, with $\eta_L$ being the smallest solution of the fixed point equation $\eta_L=G_{\underline{D_{\cR}(\eps)}}(G_{\underline{D_{\cL}(\eps)}}(\eta_L))$, and further, $\xi_R=1-G_{D_{\cR}(\eps)}(\eta_R)$, with $\eta_R:=G_{\underline{D_{\cL}(\eps)}}(\eta_L)$.

    For any $t>0$, let us denote by $(D_{\cL}(t),D_{\cR}(t))$ the (random) bi-degree sequence of the bipartite multigraph $B^*_t$. Thus $D_{\cL}(t)=(d_u(t))_{u \in \cL}$ and $D_{\cR}(t)=(d_v(t))_{v \in \cR}$. We note that
    \begin{align*}
        &\Prob{\cA_{\delta}\;\;\text{holds for}\;\;B_t}\\& \leq B\cdot \Prob{\cA_{\delta}\;\;\text{holds for}\;\;B^*_t}+o(1)&\text{(Theorem \ref{prop:multi_approx})}\\&\leq B\cdot \Exp{\CProb{\cA_{\delta}\;\;\text{holds for}\;\;B^*_t}{D_{\cL}(t),D_{\cR}(t)}}+o(1)\\& \leq B\cdot \Prob{\cA_{\delta}\;\;\text{holds for}\;\;\mathrm{BCM}(D_{\cL}(t),D_{\cR}(t))}+o(1).&\hspace{10 pt}\text{(Proposition \ref{prop:BCM_approx})}\numberthis \label{eq:RHS_giant}   
    \end{align*}
As we have already observed, the random bi-degree sequence $(D_{\cL}(t),D_{\cR}(t))$ satisfies Assumption \ref{assump:giant_BCM} with limiting variables $D_{\cL}(\eps)$ and $D_{\cL}(\eps)$. Observe that thanks to \eqref{eq:left_eq_right}, $\gamma={\Exp{D_{\cL}(\eps)}/}{\Exp{D_{\cR}(\eps)}}$ is in fact the limit of $|\cR|/|\cL|=R_n/L_n$. Also, $p_2+q_2=\Prob{D_{\cL}(\eps)=2}+\Prob{D_{\cR}(\eps)=2}<2$ by \eqref{eq:neg_bin_pmf}. Next, we verify the supercriticality condition $\Exp{\underline{D_{\cL}(\eps)}}\Exp{\underline{D_{\cR}(\eps)}}>1$. Note that the condition can be equivalently written as,
\begin{align*}
    \frac{\Exp{D_{\cL}(\eps)(D_{\cL}(\eps)-1)}}{\Exp{D_{\cL}(\eps)}}\frac{\Exp{D_{\cR}(\eps)(D_\cR(\eps)-1)}}{\Exp{D_\cR(\eps)}}>1.
\end{align*}
From Lemma \ref{lem:nb_properties} recall that if $X \sim \mathrm{NB}(\alpha,p)$, then $\Exp{X}=\frac{\alpha p}{1-p}$ and $\Exp{X(X-1)}=\left(\frac{\alpha p}{1-p} \right)^2(1+\alpha^{-1})$. Also recall the definition of $t_c$ from \eqref{eq:t_c_def}. Recalling the definitions of the variables $D_\cL(\eps)$ and $D_{\cR}(\eps)$ from \eqref{prop:neg_bin_left_right_deg_seq}, some easy algebra then yields that the last condition is equivalent to $t>t_c(L_n+R_n)$, which is true as $\eps>0$.

Finally, using Lemma \ref{lem:nb_properties} again, and some easy rewriting, it can be verified that the description of $\xi_L,\xi_R,\eta_L,\eta_R$ as given in Theorem \ref{thm:giant} matches with the description in Theorem \ref{thm:BCM_PT_vdH}, with $D_\square$ replaced by $D_{\square}(\eps)$ for $\square \in \{\cL,\cR\}$. Thus, applying Theorem \ref{thm:BCM_PT_vdH}, the RHS of \eqref{eq:RHS_giant} is $o(1)$.
\end{proof}
In the following sections, we provide the proofs of Propositions \ref{prop:multi_approx}, \ref{prop:BCM_approx} and \ref{prop:neg_bin_left_right_deg_seq}. 

\subsection{Proof of Proposition \ref{prop:multi_approx}}\label{sec:multi_approx}
For any bipartite (multi or simple) graph $H$ on the vertex bipartition $V=\cL \cup \cR$, where any vertex $u$ has degree $d_u(H)$, denote 
\begin{align*}
Q(H)=\sum_{u \in \cL}d_u(H)^3+\sum_{v \in \cR}d_v(H)^3.\numberthis \label{eq:def_Q}   
\end{align*}
\begin{lemma}[Measure change lemma]\label{lem:measure_change}
    Consider a deterministic sequence of bipartite multigraphs $(H_i)_{0 \leq i \leq t}$ on the vertex set $V=\cL \cup \cR$, with the property that $H_{i+1}$ is obtained from $H_i$ by including a single edge. Then, denoting $\eta=\alpha^2+\beta^2$, we have
        \begin{align*}
    \Prob{(B_j)_{0\leq j \leq t}=(H_j)_{0\leq j \leq t}}\leq \prod_{i=0}^{t-1}\left(1-\frac{4(Q(H_i)+\eta i)}{(i+\alpha L_n)(i+\beta R_n)} \right)^{-1}\cdot \Prob{(B^*_j)_{0\leq j \leq t}=(H_j)_{0\leq j \leq t}}. \\\numberthis \label{eq:exact_RN_form}
\end{align*}
\end{lemma}
\begin{proof}
    Let us denote by $\{u_i,v_i\}$ the unique edge by which $H_i$ differs from $H_{i+1}$, where $u_i \in \cL$ and $v_i \in \cR$. Observe,
\begin{align*}
    &\Prob{(B_j)_{0\leq j \leq t}=(H_j)_{0\leq j \leq t}}\\&=\prod_{i=0}^{t-1}\frac{(i+\alpha L_n)(i+\beta R_n)}{\sum_{\{u,v\}\notin E(H_i)}(d_{u}(H_i)+\alpha)(d_{v}(H_i)+\beta)}\Prob{(B^*_j)_{0\leq j \leq t}=(H_j)_{0\leq j \leq t}}\\&=\prod_{i=0}^{t-1}\left(1-\frac{\sum_{\{u,v\}\in E(H_i)}(d_u(H_i)+\alpha)(d_v(H_i)+\beta)}{(i+\alpha L_n)(i+\beta R_n)} \right)^{-1}\Prob{(B^*_j)_{0\leq j \leq t}=(H_j)_{0\leq j \leq t}}.
\end{align*}

Note that for any $x,y>0$, $(x+\alpha)(y+\beta)\leq 4x^2+4y^2+4\eta$. Thus,
\begin{align*}
    \Prob{(B_j)_{0\leq j \leq t}=(H_j)_{0\leq j \leq t}}\leq \prod_{i=0}^{t-1}\left(1-\frac{4(Q(H_i)+\eta i)}{(i+\alpha L_n)(i+\beta R_n)} \right)^{-1}\cdot \Prob{(B^*_j)_{0\leq j \leq t}=(H_j)_{0\leq j \leq t}},
\end{align*}
finishing the proof.
\end{proof}
We can now prove Proposition \ref{prop:multi_approx}.

\begin{proof}[Proof of Proposition \ref{prop:multi_approx}]
Since $i \leq t \leq C(L_n+R_n)$ and by \eqref{eq:assump_both_pos_prop} $L_n=\Theta(R_n)$, we immediately have from the last display that for any constant $A>0$ and any set $\cA$ of graph sequences $(H_i)_{0 \leq i \leq t}$,
\begin{align*}
    \Prob{(B_j)_{0\leq j \leq t}\in \cA}\leq \Theta(1)\cdot \Prob{(B^*_j)_{0\leq j \leq t}\in \cA}+\Prob{\text{for some}\;\;i\leq t-1,\;\;Q(B_i)>A(L_n+R_n)}, \\\numberthis \label{eq:mult_approx_bootstrap_UB}
\end{align*}
(where the $\Theta(1)$ term depends on the constant $A$) so that it is enough to check that the second term on the RHS above is $o(1)$ for some $A$ sufficiently large. A union bound, and using the facts that $Q(B_0)=0$ and $Q(B_t)\geq Q(B_{t'})$ if $t\geq t'$, gives this quantity is at most
\begin{align*}
    \sum_{i=0}^t\Prob{Q(B_{i-1})<A(L_n+R_n)\;\;\text{while}\;\;Q(B_i)>A(L_n+R_n)}
\end{align*}
which using that $t \leq C(L_n+R_n)$ and \eqref{eq:mult_approx_bootstrap_UB} is at most (again using the monotonicity of $Q$)
\begin{align*}
    \Theta(1)\cdot C(L_n+R_n)\cdot \Prob{Q(B^*_t)>A(L_n+R_n)}.
\end{align*}
Note that $A$ can be chosen a large enough constant such that $\Prob{Q(B^*_t)>A(L_n+R_n)}=o((L_n+R_n)^{-1})$. This is possible using the second assertion of Proposition \ref{prop:neg_bin_left_right_deg_seq} for some large $k$ ($k \geq 5$ suffices) combined with a Markov's inequality, and noting that since $t\leq C(L_n+R_n)$, both $\Exp{D_{n,L}^k},\Exp{D_{n,R}^k}=O(1)$ (with the $O(1)$ term depending on $C$) for any $k \geq 1$.
\end{proof}
\subsection{Proof of Proposition \ref{prop:BCM_approx}}
\subsubsection{A convenient construction of the BCM}\label{sec:conv_cons_bcm}
Recall the definition of the BCM from Section \ref{sec:giant_proof}. Consider a bi-degree sequence $\bD=(\bD_{\cL},\bD_{\cR})$ on a bipartitioned vertex set $V=\cL \cup \cR$, where $\bD_\cL=(d_u)_{u \in \cL}$ and $\bD_\cR=(d_v)_{v \in \cR}$ satisfies $m=\sum_{u \in \cL}d^{(L)}_u=\sum_{v \in \cL}d^{(R)}_v.$ The following is a convenient way to construct the BCM with bi-degree sequence $\bD$. 
\begin{itemize}
    \item Let $\cS(\bD_\cL)$ be the set of all sequences of the form $\mathbf{u}:=(u_1,u_2,\dots,u_m)$ where each $u_i \in \cL$, and any $u \in \cL$ appears exactly $d_u$ many times in $\mathbf{u}$. Similarly, let $\cS(\bD_\cR)$ be the set of all sequences $\mathbf{v}=(v_1,\dots,v_m)$ with each $v \in \cR$ appearing exactly $d_v$ many times in $\mathbf{v}$.
    \item Sample uniformly at random $\mathbf{U}=(U_1,\dots,U_m)$ from $\cS(\bD_L)$, and independently, uniformly at random $\mathbf{V}=(V_1,\dots,V_m)$ from $\cS(\bD_R)$. Include the edges $\{U_i,V_i\}$ for all $1\leq i \leq m$ to form a bipartite multigraph on $V=\cL \cup \cR$, and note that it has the same distribution as $\mathrm{BCM}(\bD_L,\bD_R)$.
\end{itemize}
\subsubsection{A convenient construction of the multigraph process}\label{sec:conv_cons_multi}
Observe the following related construction of the multigraph $(B^*_t)_{t \geq 0}$. The product structure on the right hand side of \eqref{eq:multi_prob} implies the following independence that we exploit. Given $B^*_{t-1}$, sample a vertex $u \in \cL$ with probability $\frac{(d_u(t-1)+\alpha)}{\sum_{x \in \cL}(d_x(t-1)+\alpha)}$, and independently sample $v \in \cR$ with probability $\frac{(d_v(t-1)+\beta)}{\sum_{y \in \cR}(d_y(t-1)+\beta)}$. Include the edge $\{u,v\}$ in $B^*_{t-1}$ to obtain $B^*_t$.

Keeping this observation in mind, we introduce the following random sequence of vertices from $\cL$ and $\cR$. Let $X_1$ and $Y_1$ be independently uniformly sampled from respectively $\cL$ and $\cR$, and for $i \geq 1$, given $X_1,\dots,X_i$ and $Y_1,\dots, Y_i$, sample $X_{i+1} \in \cL$ and $Y_{i+1}\in \cR$ conditionally independently with probability masses
\begin{align*}
    \CProb{X_{i+1}=u}{(X_j)_{j=1}^i}=\frac{\sum\limits_{1\leq j \leq i}(\ind{X_j=u}+\alpha)}{i+\alpha L_n}\;\;\text{and}\;\; \CProb{Y_{i+1}=v}{(Y_j)_{j=1}^i}=\frac{\sum\limits_{1\leq j \leq i}(\ind{Y_j=v}+\beta)}{i+\beta R_n}.
\end{align*}
Note that $\{\{X_i,Y_i\}:1\leq i \leq t\}$ has the same distribution as the edge set of $B^*_t$.

We can now provide the proof of Proposition \ref{prop:BCM_approx}. Consider a bi-degree sequence $\bD=(\bD_\cL,\bD_\cR)$, and recall the sets $\cS(\bD_\cL)$ and $\cS(\bD_\cR)$ from Section \ref{sec:conv_cons_bcm}. Let us use the notation $\mathrm{BP}(\mathbf{u},\mathbf{v})$ to denote the bipartite multigraph on $V=\cL \cup \cR$ with edge set $\{u_i,v_i\}$, for fixed $\mathbf{u} \in \cS(\bD_\cL),\mathbf{v} \in \cS(\bD_\cR)$.

\begin{proof}[Proof of Proposition \ref{prop:BCM_approx}]
 Thanks to the constructions described in Sections \ref{sec:conv_cons_bcm} and \ref{sec:conv_cons_multi}, for a given bipartite multigraph $G$ with bi-degree sequence $(\bD_\cL,\bD_\cR)$ we have
\begin{align*}
    &\CProb{B^*_t=G}{(D_\cL{(t)},D_\cR{(t)})=(\bD_\cL,\bD_\cR)}\\&=\frac{\Prob{B^*_t=G}}{\Prob{(D_\cL{(t)},D_\cR{(t)})=(\bD_\cL,\bD_\cR)}}\\&=\frac{\sum\limits_{\mathbf{u}\in \cS(\bD_\cL), \mathbf{v}\in \cS(\bD_\cR): \atop \mathrm{BP}(\mathbf{u},\mathbf{v})=G}\Prob{X_i=u_i\;\;\forall\;\;1\leq i \leq t}\Prob{Y_i=v_i\;\;\forall\;\;1\leq i \leq t}}{\left(\sum\limits_{\mathbf{u}\in \cS(\bD_\cL)}\Prob{X_i=u_i\;\;\forall\;\;1\leq i \leq t}\right)\left(\sum\limits_{\mathbf{v}\in \cS(\bD_\cR)}\Prob{Y_i=v_i\;\;\forall\;\;1\leq i \leq t}\right)}\\&=\frac{\sum\limits_{\mathbf{u}\in \cS(\bD_\cL),\mathbf{v}\in \cS(\bD_\cR): \atop \mathrm{BP}(\mathbf{u},\mathbf{v})=G}\frac{\prod_{u \in \cL}\prod_{p=0}^{d_u}(p+\alpha)}{\prod_{j=0}^{t-1}(j+\alpha L_n)}\frac{\prod_{v \in \cR}\prod_{p=0}^{d_v}(p+\beta)}{\prod_{j=0}^{t-1}(j+\beta R_n)}}{\left(\sum\limits_{\mathbf{u}\in \cS(\bD_\cL)}\frac{\prod_{u \in \cL}\prod_{p=0}^{d_u}(p+\alpha)}{\prod_{j=0}^{t-1}(j+\alpha L_n)}\right)\left(\sum\limits_{\mathbf{v}\in \cS(\bD_\cR)}\frac{\prod_{v \in \cR}\prod_{p=0}^{d_v}(p+\beta)}{\prod_{j=0}^{t-1}(j+\beta R_n)}\right)},
\end{align*}
the last line following by the definition of the random variables $(X_i)_{1\leq i \leq t}$ and $(Y_i)_{1\leq i \leq t}$. Note that the denominator in the last expression is simply a function of the bi-degree sequence of $G$, and so is the summand in the numerator. Thus, the last probability is simply the probability of the event $\{\mathrm{BP}(\mathbf{U},\mathbf{V})=G\}$, where $\mathbf{U},\mathbf{V}$ are respectively uniformly distributed independent random elements from $\cS(\bD_\cL)$ and $\cS(\bD_\cR)$. Recalling the construction of the $\mathrm{BCM}(\bD_\cL,\bD_\cR)$ from Section \ref{sec:conv_cons_bcm}, we thus obtain that conditional on the event $\{(D_\cL(t),D_\cR(t))=(\bD_\cL,\bD_\cR)\}$, $B^*_t$ is distributed as a $\mathrm{BCM}(\bD_\cL,\bD_\cR)$. 
\end{proof}
\subsection{Proof of Proposition \ref{prop:neg_bin_left_right_deg_seq}}
\subsubsection{Continuous-time embedding}\label{sec:cts_time}
The following continuous-time embedding is due to Janson and Warnke \cite{janson2021preferential}. Recall the sequence of random variables $(X_i)_{i=1}^t$ and $(Y_i)_{i=1}^t$ from section \ref{sec:conv_cons_multi}, and note that for any $u \in \cL$ and $v \in \cR$, $d_u(t)\stackrel{d}{=} \sum_{i=1}^t\ind{X_i=u}$ and $d_v(t)\stackrel{d}{=} \sum_{i=1}^t\ind{Y_i=v}$. To sample the random variables $(X_i)_{i=1}^t$ and $(Y_i)_{i=1}^t$, it is thus natural to introduce independent pure birth processes $(D_u(r))_{u \in \square}$ for $\square \in \{\cL,\cR\}$, where recalling the map $\rho$ from \eqref{eq:maps_rho_zeta}, for each $u \in \square$, the process $(D_u(r))_{r \geq 0}$ gives the next birth at rate $k+\rho(\square)$ given that it has already given $k$ births, $k \geq 0$. Then if we introduce stopping times $\tau_{\square,1},\tau_{\square,2},\dots$ where the $i$-th birth in the union of all the processes $(D_u(r))_{u \in \square}$ takes place at time $\tau_{\square,i}$, we note that $X_i$ has the same distribution as the vertex in $\cL$ that has given birth at time $\tau_{\cL,i}$, and similarly $Y_i \in \cR$ has the same distribution as the vertex that has given birth at time $\tau_{\cR,i}$. For fixed $r\geq 0$ and any $\square \in \{\cL,\cR\}$ it is standard that the distribution of the number of points in $D_u(r)$ for any $u\in \square$ is distributed as $\mathrm{NB}(\rho(\square),1-e^{-r})$, see e.g., \cite[Section 3.2]{janson2021preferential}. The following lemma is essentially \cite[Theorem 2.5]{janson2021preferential}. We have stated the result slightly more generally for our purposes, letting the number of edges to be superlinear and providing an explicit error bound, but the same proof works with minor modifications. 
\begin{lemma}[Stopping time lemma]\label{lem:stopping_time}
    Let $t=t_n$ be a sequence of $n$ satisfying $\liminf\limits_{n \to \infty}\frac{t}{L_n+R_n} \in (0,\infty]$. For any $\square \in \{\cL,\cR\}$, and for any sequence $s=s_n$ satisfying $s_n \leq t,$ we have 
    \begin{align*}
        \Prob{\tau_n^{\square,-}(t,s)\leq \tau_{\square,t} \leq \tau_n^{\square,+}(t,s)}\geq 1-\left(\frac{4}{s^2}\left(\frac{t^2}{\rho(\square) |\square|}\right) + t) \right),
    \end{align*}
    where \begin{align*}
        \tau_n^{\square,\pm}(t,s):=\log \left(1+\frac{t\pm s}{\rho(\square) |\square|} \right). \numberthis \label{eq:def_ub_lb_stop_times}
    \end{align*}
    In particular, $\liminf_{n \to \infty}\Prob{\tau_n^{\square,-}(t,s)\leq \tau_{\square,t} \leq \tau_n^{\square,+}(t,s)} =1$, if $\frac{t}{\sqrt{L_n+R_n}}\ll s_n \leq t$.
\end{lemma}
The proof of Proposition \ref{prop:neg_bin_left_right_deg_seq} is now straightforward using Lemma \ref{lem:stopping_time}:
\begin{proof}[Proof of Proposition \ref{prop:neg_bin_left_right_deg_seq}]
    Assume $t=\Theta(L_n+R_n)$. For $\square \in \{\cL,\cR\}$, note that we have to conclude about $D_{U_\square}(\tau_{\square,t})$, where $U_\square$ is uniformly distributed on $\square$. But thanks to Lemma \ref{lem:stopping_time}, in fact for any $u \in \square$, using the monotonicity of $D_u(\cdot)$, \textbf{whp}
    \begin{align*}
        D_u(\tau_n^{\square,-}(s,t)) \leq D_u(\tau_{\square,t}) \leq D_u(\tau_n^{\square,+}(s,t)),
    \end{align*}
    where say $s=s_n=\ln(L_n+R_n)\sqrt{L_n+R_n}$. As for any $r>0$ we have $D_u(r)\sim \mathrm{NB}(\rho(\square),1-e^{-r})$, it is a matter of simple verification that the upper and lower bounding variables in the above display satisfy the conclusions of the Proposition, and hence so does the middle variable. 
\end{proof}

\section{Proof of connectivity properties}\label{sec:connectivity properties}
In this section we give the proofs of Theorems \ref{thm:multi_conn} and \ref{thm:sg_LB}. We begin with a result on the number of isolated vertices. 
\begin{lemma}[Number of isolated vertices]\label{lem:iso_vertices}
    Fix $\square \in  \{\cL,\cR\}$. Let $t=t_n$ satisfy $\lim\limits_{n \to \infty}\frac{t_n}{(L_n+R_n)^{1+\rho(\square)^{-1}}}=x \in [0,\infty]$. Denote by $\cI(\square,t)$ the number of isolated vertices from $\square$ in the graph $B^*_t$.
    \begin{itemize}
        \item If $x=0$, $\liminf\limits_{n \to \infty}\Prob{\cI(\square,t)\geq 1}\to 1$.
        \item If $x \in (0, \infty)$, as $n \to \infty$, $\cI(\square,t) \dlim \mathrm{Poi}\left(\left(\frac{\rho(\square)}{x}\right)^{\rho(\square)}(\zeta(\square))^{-1-\rho(\square)} \right).$
    \item If $x=\infty$, $\limsup\limits_{n \to \infty}\Prob{\cI(\square,t)\geq 1}=0$.
    \end{itemize}
     
\end{lemma}
\begin{proof}
    We work with the independent continuous-time processes $(D_u(r))_{r \geq 0}$ for $u \in \square$. Note that denoting by $\cB_u(s)$ the event that no birth has taken place in the process $(D_u(r))_{r \geq 0}$ up to time $s\geq 0$, we need to conclude about the random variable $\sum_{ u\in \square}\indE{\cB_u(\tau_{\square,t})}$. Using Lemma \ref{lem:stopping_time} and the monotonicity of the processes $D_u(\cdot)$, \textbf{whp}
    \begin{align*}
        \sum_{ u\in \square}\indE{\cB_u(\tau_{n}^{\square,+}(t,s))}\leq \sum_{ u\in \chi(\eta)}\indE{\cB_u(\tau_{\square,t})}\leq \sum_{ u\in \square}\indE{\cB_u(\tau_{n}^{\square,-}(t,s))}, \numberthis \label{eq:iso_dom_above_below}
    \end{align*}
    where let $s=(L_n+R_n)^{3/4+\rho(\square)^{-1}}$. As the first birth in any of the processes $(D_u(r))_{r \geq 0}$ takes place after an $\mathrm{Exp}(\rho(\square))$ amount of time,
    \begin{align*}
        \sum_{ u\in \square}\indE{\cB_u(\tau_{n}^{\square,\pm}(t,s))}\sim \mathrm{Bin}\left(|\square|,p_{\pm} \right),\;\;\text{with}\;\; p_{\pm}=\exp\left(-\rho(\square) \log \left(1+\frac{t\pm s}{|\square|\rho(\square)} \right) \right). \numberthis \label{eq:binomial_isolated}
    \end{align*}
    
    Using \eqref{eq:iso_dom_above_below}, it is thus enough to prove $\Prob{\sum_{ u\in \square}\indE{\cB_u(\tau_{n}^{\square,-}(t,s))} \geq 1} \to 1$ to conclude the first assertion of the lemma. Since $\rho(\square)>0$ and $t\ll |\square|^{1+\rho(\square)^{-1}}$, let us write $1+\frac{t-s}{|\square|\rho(\square)}=(s_n |\square|)^{\rho(\square)^{-1}}$ where $s_n=o(1)$. Using \eqref{eq:binomial_isolated}, note that $\Exp{\sum_{ u\in \square}\indE{\cB_u(\tau_{n}^{\square,-}(t,s))}}=s_n^{-1}\gg 1$, and $\mathrm{Var}\left( \sum_{ u\in \square}\indE{\cB_u(\tau_{n}^{\square,-}(t,s))}\right)=\frac{1}{s_n}\left(1-\frac{1}{s_n |\square|} \right)$, so that a standard second-moment method finishes the proof.

    The proof of the second and the third assertions are similar. For the second one, we simply note that the upper and lower bounding variables in \eqref{eq:iso_dom_above_below} converge in distribution to the required Poisson variable, thus so does the middle one. For the third one, a mean computation shows that $\Exp{\sum_{ u\in \square}\indE{\cB_u(\tau_{n}^{\square,+}(t,s))}}=o(1)$, and the proof is finished using \eqref{eq:iso_dom_above_below} and Markov's inequality. 
\end{proof}
\begin{remark}
    In \cite{pittel2010random}, Pittel proves a similar Poisson approximation statement for the number of isolated vertices in the unipartite case using combinatorial techniques. We instead note the usefulness of the continuous-time embedding of Section \ref{sec:cts_time}. We believe using this continuous description, one should also be able to prove similar Poisson approximation statements for counts of \emph{small} subgraphs, in particular, identify their appearance thresholds. This is something we do not address in this paper and leave for future work.
\end{remark}
We continue by deriving useful expressions for the probabilities of certain events of interest. The analysis of combinatorial flavor that follows is inspired by Pittel \cite{pittel2010random}. For a bipartite multigraph $M$ on vertex bipartition $V=\cL \cup \cR$, and for any $A\subset \cL$, denote the \emph{neighborhood} of $A$ in $M$ by
\begin{align*}
    N_M(A):=\{v \in \cR:\exists\; u \in A \;\;\text{satisfying}\;\;\{u,v\}\in E(M)\},
\end{align*} where $E(M)$ is the multiset of edges in $M$. For convenience, we write $\cN_t(A)=\cN(A,B^*_t)$ for any $t \geq 0$. Let $\cM_t$ be the set of all bipartite multigraphs on $V=\cL \cup \cR$ with $t$ edges, and further for any $A\subset \cL$, $B \subset \cR$, $t_1,y \geq 0$ with $t_1+y\leq t$, let
\begin{align*}
    \cB_t(A,B,t_1,y):=\left\{M \in \cM_t: \cN(A,M)\subset B, \sum_{u \in A}d_u(M)=t_1, \sum_{v \in B}d_v(M)=t_1+y\right\}, \numberthis \label{eq:def_B_t_event}
\end{align*}
where $d_u(M)$ denotes the degree of $u\in \cL$ in the multigraph $M$. For any $\alpha >0$ and $m \in \N$, we denote the \emph{rising factorial} as $(\alpha)_m:=\alpha(\alpha+1)\dots(\alpha+m-1).$

\begin{proposition}[Edge-partition formula]\label{prop:edge_part_form}
Letting by convention $\binom{t-t_1}{0}=1$,
    \begin{align*}
        &\Prob{\cG^*_t\in \cB_t(A,B,t_1,y)}\\&={\binom{t}{t_1}\binom{t-t_1}{y}}\cdot \frac{(|A|\alpha)_{t_1}((L_n-|A|)\alpha)_{t-t_1}}{(\alpha L_n)_{t}}\cdot \frac{(|B|\beta)_{t_1+y}((R_n-|B|)\beta)_{t-t_1-y}}{(\beta R_n)_{t}}.\numberthis \label{eq:edge_part_formula}
    \end{align*}
\end{proposition} We need a couple of useful tools to prove this proposition, which we state as the following two lemmas. 
\begin{lemma}\label{lem:graph_prob}
For any bipartite multigraph $M$ on $V=\cL \cup \cR$, where any $u \in \cL, v \in \cR$ has exactly $n_M(u,v)$ many edges between them, the degree of any vertex $u$ is $d_u(M)$, and the total number of edges $\sum_{u \in \cL, v \in \cR}n_M(u,v)=t$,
\begin{align*}
    \Prob{B^*_t=M}=\frac{\prod_{u \in \cL}(\alpha)_{d_u(M)}}{(\alpha L_n)_t}\cdot \frac{\prod_{v \in \cR}(\beta)_{d_v(M)}}{(\beta R_n)_t}\cdot \frac{t!}{\prod_{u \in \cL, v \in \cR}n_M(u,v)!}.
\end{align*}    
\end{lemma}
\begin{proof} The combinatorial factor $\frac{t!}{\prod_{u \in \cL, v \in \cR}n_M(u,v)!}$ counts the number of \emph{graph paths} $(G_0,G_1,\dots,G_t)$ where $G_0$ is the empty graph, $G_t=M$, and $G_{i+1}$ is obtained from $G_i$ by the inclusion of exactly one edge for all $0\leq i \leq t-1$. Finally, for a given such graph path $(G_0,\dots, G_t)$, recalling the convenient construction of $B^*_t$ from Section \ref{sec:conv_cons_bcm}, we have
\begin{align*}
    \Prob{B^*_j=G_j\;\;\forall\;\;0\leq j \leq t}=\frac{\prod_{u \in \cL}(\alpha)_{d_u(M)}}{(\alpha L_n)_t}\cdot \frac{\prod_{v \in \cR}(\beta)_{d_v(M)}}{(\beta R_n)_t}.
\end{align*}\end{proof}
\begin{lemma}
    For any integer $n \geq 1$, consider reals $p_i\geq 0$ for $1\leq i \leq n$ and let $\mathbf{Y}_k$ denote the set of all non-negative integer sequences $\mathbf{y}=(y_i)_{1\leq i \leq n}$ satisfying $\sum_{1\leq i \leq n}y_i=k$ for some integer $k \geq 0$. Then,
    \begin{align*}
        \sum_{\mathbf{y}\in \mathbf{Y}_k}\prod_{i=1}^n\frac{(p_i)_{y_i}}{y_i!}=\frac{(\sum_{i=1}^n p_i)_k}{k!}\numberthis \label{eq:nb_thm_conseq}
    \end{align*}
\end{lemma}
\begin{proof}
    Recall the negative binomial theorem: $\sum_{p \geq 0} \frac{(\alpha)_p}{p!}z^p=(1-z)^{-\alpha}$ holds for any $z \in \R$ with $|z|<1$. As a consequence, the quantity on the LHS of \eqref{eq:nb_thm_conseq} is simply the coefficient of $z^{k}$ in the expansion of $\prod_{i=1}^n (1-z)^{-p_i}=(1-z)^{-(\sum_{i=1}^n p_i)}$. 
\end{proof}
We can now provide the proof of Proposition \ref{prop:edge_part_form}. 
\begin{proof}[Proof of Proposition \ref{prop:edge_part_form}]
    To conclude Proposition \ref{prop:edge_part_form}, we need to sum the probability formula provided by Lemma \ref{lem:graph_prob}, over all multigraphs in $\cB_t(A,B,t_1,y)$. Note that for any fixed $M \in \cB_t(A,B,t_1,y)$,
    \begin{align*}
        &\Prob{B^*_t=M}\\&={\binom{t}{t_1}\binom{t-t_1}{y}}\cdot\frac{\prod_{u \in \cL}(\alpha)_{d_u(M)}}{(\alpha L_n)_m}\cdot \frac{t_1!}{\prod_{u \in A, v \in B}n_M(u,v)!}\\&\hspace{75 pt} \cdot \frac{\prod_{v \in \cR}(\beta)_{d_v(M)}}{(\beta R_n)_m}\cdot \frac{y!}{\prod_{u \in A^c, v \in B}n_M(u,v)!}\cdot \frac{(t-t_1-y)!}{\prod_{u \in A^c, v \in B^c}n_M(u,v)!},\numberthis \label{eq:form_M_edge_part}
    \end{align*}
    where $A^c=\cL \setminus A$ and $B^c = \cR \setminus B$. Recall the sets $\cS(\bD_\cL)$ and $\cS(\bD_\cR)$ from Section \ref{sec:conv_cons_bcm}, where $\bD_\cL=(d_u(M))_{u \in \cL},\bD_\cR=(d_v(M))_{v \in \cR}$, and where for any $u \in \cL, v \in \cR$ we have the degrees $d_u(M):=\sum_{v \in \cR}n_M(u,v), d_v(M):=\sum_{u \in \cL}n_M(u,v)$. Note that for any fixed $M \in \cB_t(A,B,t_1,y)$, 
\begin{align*}
    &|\{(\mathbf{u},\mathbf{v})\in \cS(\bD_\cL)\times \cS(\bD_\cR): \mathrm{BP}(\mathbf{u},\mathbf{v})=M\}|\\&=\frac{t!}{\prod_{u \in \cL v \in \cR}n_M(u,v)!}\\&=\binom{t}{t_1}\binom{t-t_1}{y}\cdot \frac{t_1!}{\prod_{u \in A, v \in B}n_M(u,v)!}\cdot \frac{y!}{\prod_{u \in A^c, v \in B}n_M(u,v)!}\cdot \frac{(t-t_1-y)!}{\prod_{u \in A^c, v \in B^c}n_M(u,v)!}, \numberthis \label{eq:form_seq_pair_giving_M}
\end{align*}
where the first equality above is true simply because given a fixed $(\mathbf{u},\mathbf{v})\in \{(\mathbf{u},\mathbf{v})\in \cS(\bD_\cL)\times \cS(\bD_\cR): \mathrm{BP}(\mathbf{u},\mathbf{v})=M\}$, all such other $(\mathbf{u},\mathbf{v})$ can be obtained by permuting the edges $\{(u_i,v_i)\}_{1\leq i \leq t}$, where the multi-edges between any pair of vertices are considered identical.

For any $\square \in \{\cL,\cR\}$, $H\subset \square$, $m\geq 0$, and a non-negative integer sequence $\mathbf{d}=(d_h)_{h \in H}$ satisfying $\sum_{h \in H}d_h=m$, we write $\cS(H,\mathbf{d},m)$ to be the set of all sequences $(w_1,\dots,w_m)$ with each $w_i \in H$, and with the property that any $h \in H$ has multiplicity $d_h$ in the multiset $\{w_1,\dots,w_m\}$. This is essentially the analog of $\cS(\cD_\square)$, but restricted to the vertex subset $H$. Observe that
\begin{align*}
|\cS(H,\mathbf{d},m)|=\frac{m!}{\prod_{h \in H}d_h!}.\numberthis \label{eq:restricted_seq_count}
\end{align*}

Consider the size of the set $\{(\mathbf{u},\mathbf{v})\in \cS(\bD_\cL)\times \cS(\bD_\cR): \mathrm{BP}(\mathbf{u},\mathbf{v})=M\}$ that appears in the LHS of \eqref{eq:form_seq_pair_giving_M}. Sum this size over all multigraphs $M$ with the following properties:
\begin{properties}\label{proper:sum} \hspace{100 pt}
 \begin{enumerate}    \item[1.]\label{prop:one} Bi-degree sequence $(\bD_\cL,\bD_\cR)$.
    \item[2.]\label{prop:two} $\cN(A,M)\subset B$.
    \item[3.]\label{prop:three} $\sum_{u \in A} d_u(M)=t_1$.
    \item[4.]\label{prop:four} $\sum_{v \in B}d_v(M)=t_1+y$. 
    \item[5.]\label{prop:five} Degree sequence from $A$ is $\bD(A)=(d(A)_v)_{v \in \cR}$. This simply means for each $v \in \cR$, it has exactly $d(A)_v$ many neighbors from $A$ incident to it. Note that by (2.) above, for any $v \in \cR \setminus B$, we must have $d(A)_v=0$, and by (3.) above, we must have $\sum_{v \in B}d(A)_v=t_1$. 
\end{enumerate}
\end{properties}

 The obtained sum thus gives the size of the set of all pairs $(\mathbf{u},\mathbf{v}) \in \cS(\bD_\cL)\times \cS(\bD_\cR)$ such that $BP(\mathbf{u},\mathbf{v}) \in \cB_t(A,B,t_1,y)$ with bi-degree sequence $(\bD_\cL,\bD_\cR)$ and degree sequence of $B$ from $A$ being $\bD(A)$.

 On the other hand, for any $U \subset \square \in \{\cL,\cR\}$, denoting the restriction of the sequence $\bD_\square$ to $U$ by $\bD_{\square,U}$ and using the notation $\bD(A)^c$ to denote the sequence $\bD_{\cR,B}-\bD(A)=(d_v(M)-d(A)_v)_{v \in \cR}$, we note that this size also has to equal 
\begin{align*}
&\binom{t}{t_1}\binom{t-t_1}{y}\cdot \left|\left(\cS(A,\bD_{\cL,A},t_1)\times\cS(A^c,\bD_{\cL,A^c},t-t_1) \right)\right|\\&\hspace{62 pt} \cdot\left| \left(\cS(B,\bD(A)^c,y)\times \cS(B,\bD(A),t_1)\times \cS(B^c,\bD_{\cR,B^c},t-t_1-y)\right)\right|.  \numberthis \label{eq:counting_size_second}  
\end{align*}
Here the first binomial coefficient $\binom{t}{t_1}$ above simply chooses the coordinate indices of $\mathbf{u}$ corresponding to vertices from $A$. Since we require $\cN(A,M)\subset B$, the same indices in $\mathbf{v}$ correspond to coordinates that are vertices from $B$. However, $B$ has exactly $y$ incidences from $A^c$. Thus the second binomial coefficient $\binom{t-t_1}{y}$ above chooses the coordinate indices of $\mathbf{v}$ other than the already chosen ones, where a vertex from $B$ can be present.

In other words, using \eqref{eq:counting_size_second}, and recalling \eqref{eq:restricted_seq_count} and \eqref{eq:form_seq_pair_giving_M}, the following identity is true:
\begin{align*}
    &\sum_{M}\frac{t!}{\prod_{u \in \cL v \in \cR}n_M(u,v)!}\\&=\sum_{M}\binom{t}{t_1}\binom{t-t_1}{y}\cdot \frac{t_1!}{\prod_{u \in A, v \in B}n_M(u,v)!}\cdot \frac{y!}{\prod_{u \in A^c, v \in B}n_M(u,v)!}\cdot \frac{(t-t_1-y)!}{\prod_{u \in A^c, v \in B^c}n_M(u,v)!}\\&=\binom{t}{t_1}\binom{t-t_1}{y} \cdot \left(\frac{t_1!}{\prod_{u \in A}(d_u(M))!}\cdot \frac{(t-t_1)!}{\prod_{v \in A^c}(d_u(M))!}  \right)\\&\hspace{75 pt}\cdot \left(\frac{y!}{\prod_{v \in B}(d_v(M)-d(A)_v)!}\cdot \frac{t_1!}{\prod_{v \in B}(d(A)_v)!}\cdot \frac{(t-t_1-y)!}{\prod_{v \in B^c}(d_v(M))!} \right), \numberthis \label{eq:form_after_sum_over_M}
\end{align*}
where the sum above is over all $M \in \cB_t(A,B,t_1,y)$ satisfying the five properties listed in Properties \ref{proper:sum}.
Using \eqref{eq:form_M_edge_part}, \eqref{eq:form_seq_pair_giving_M} and \eqref{eq:form_after_sum_over_M}, recalling the (random) bi-degree sequence of $B^*_t$ is $(D_\cL{(t)},D_\cR{(t)})$, and denoting by $D{(A,t)}$ the (random) degree sequence in $B^*_t$ from $A$ to vertices of $\cR$, we thus note that 
\begin{align*}
    &\Prob{B^*_t\in \cB_t(A,B,t_1,y), (D_\cL{(t)},D_\cR{(t)})=(\bD_\cL,\bD_\cR),D{(A,t)}=\bD(A) }\\&=\frac{\binom{t}{t_1}\binom{t-t_1}{y}}{(\alpha L_n)_t (\beta R_n)_t} \cdot \left(\frac{(\prod_{u \in \cL}(\alpha)_{d_u(M)})\cdot t_1! \cdot (t-t_1)!}{(\prod_{u \in A}(d_u(M))!) (\prod_{u \in A^c}(d_u(M))!)}  \right)\\&\hspace{75 pt}\cdot \left(\frac{(\prod_{v \in \cR}(\beta)_{d_v(M)})\cdot y!\cdot t_1!\cdot (t-t_1-y)!}{(\prod_{v \in B}(d_v(M)-d(A)_v)!)(\prod_{v \in B}(d(A)_v)!) (\prod_{v \in B^c}(d_v(M))!)}  \right)\\&=\frac{t!\cdot t_1!\cdot(t-t_1)!}{(\alpha L_n)_t (\beta R_n)_t}\cdot \left(\prod_{u \in A}\frac{(\alpha)_{d_u(M)}}{d_u(M)!} \right)\cdot\left(\prod_{u \in A^c}\frac{(\alpha)_{d_u(M)}}{d_u(M)!} \right)\cdot\left(\prod_{v \in B^c}\frac{(\beta)_{d_v(M)}}{d_v(M)!} \right)\\&\hspace{84 pt}\cdot\left(\prod_{v \in B}\frac{(\beta)_{d(A)_v}}{d(A)_v!} \right)\cdot\left(\prod_{v \in B}\frac{(\beta+d(A)_v)_{(d_v(M)-d(A)_v)}}{(d_v(M)-d(A)_v)!} \right).\numberthis \label{eq:to_sum_over_bideg}
\end{align*}

To conclude \eqref{eq:edge_part_formula}, we want to take a sum of the last expression over all bi-degree sequences $(\bD_\cL,\bD_\cR)$ and sequences $\bD(A)$, satisfying $\sum_{u \in A}d_u(M)=t_1=\sum_{v \in B}d(A)_v$ and $\sum_{v \in B}(d_v(M)-d(A)_v)=y$, and show that the result agrees with \eqref{eq:edge_part_formula}.

Keeping the sequence $(d(A)_v)_{v \in B}$ fixed, writing $z_v=(d_v(M)-d(A)_v)$ for $v \in B$, we note that using \eqref{eq:nb_thm_conseq}, taking a sum of \eqref{eq:to_sum_over_bideg} first over all non-negative sequences $(z_v)_{v \in B}$ with sum $\sum_{v \in B}z_v=y$ leads to
\begin{align*}
    &\frac{t!\cdot t_1!\cdot(t-t_1)!}{(\alpha L_n)_t (\beta R_n)_t}\cdot \left(\prod_{u \in A}\frac{(\alpha)_{d_u(M)}}{d_u(M)!} \right)\cdot\left(\prod_{u \in A^c}\frac{(\alpha)_{d_u(M)}}{d_u(M)!} \right)\cdot\left(\prod_{v \in B^c}\frac{(\beta)_{d_v(M)}}{d_v(M)!} \right)\\&\hspace{84 pt}\cdot\left(\prod_{v \in B}\frac{(\beta)_{d(A)_v}}{d(A)_v!} \right)\cdot\frac{(\beta|B|+\sum_{v \in B}d(A)_v)_y}{y!}.
\end{align*}
Thus, taking a sum of the last expression over all non-negative integer sequences $(d(A)_v)_{v \in B}$ with sum $\sum_{v \in B}d(A)_v=t_1$ and again using \eqref{eq:nb_thm_conseq}, we obtain
\begin{align*}
    &\frac{t!\cdot t_1!\cdot(t-t_1)!}{(\alpha L_n)_t (\beta R_n)_t}\cdot \left(\prod_{u \in A}\frac{(\alpha)_{d_u(M)}}{d_u(M)!} \right)\cdot\left(\prod_{u \in A^c}\frac{(\alpha)_{d_u(M)}}{d_u(M)!} \right)\cdot\left(\prod_{v \in B^c}\frac{(\beta)_{d_v(M)}}{d_v(M)!} \right)\cdot\frac{(\beta|B|)_{t_1}}{t_1!}\cdot\frac{(\beta|B|+t_1)_y}{y!}.
\end{align*}
Finally, employing \eqref{eq:nb_thm_conseq} again, and taking iteratively sums of the last expression over non-negative integer sequences $(d_u(M))_{u \in A}$, $(d_u(M))_{u \in A^c}$ and $(d_v(M))_{v \in B^c}$ having respectively sums equal to $t_1$, $t-t_1$ and $t-t_1-y$, we obtain
\begin{align*}
    &\frac{t!\cdot t_1!\cdot(t-t_1)!}{(\alpha L_n)_t (\beta R_n)_t}\cdot \frac{(\alpha|A|)_{t_1}}{t_1!}\cdot\frac{(\alpha(L_n-|A|))_{t-t_1}}{(t-t_1)!}\cdot\frac{(\beta(R_n-|B|))_{t-t_1-y}}{(t-t_1-y)!}\cdot\frac{(\beta|B|)_{t_1}}{t_1!}\cdot\frac{(\beta|B|+t_1)_y}{y!}\\&={\binom{t}{t_1}\binom{t-t_1}{y}}\cdot \frac{(|A|\alpha)_{t_1}((L_n-|A|)\alpha)_{t-t_1}}{(\alpha L_n)_{t}}\cdot \frac{(|B|\beta)_{t_1+y}((R_n-|B|)\beta)_{t-t_1-y}}{(\beta R_n)_{t}},
\end{align*}
finishing the proof.
\end{proof}

\begin{remark}
    As will be evident in the next section, to prove the connectivity threshold as claimed by Theorem \ref{thm:multi_conn}, we use this result for the special case $y=0$. However, we still prove this general result, as it is of independent interest. We hope to use the full strength of this result in future works where we tackle counts of small subgraphs, or the $k$-connectivity threshold.
\end{remark}

\subsection{Connectivity threshold of the multigraph process}
In this section we prove Theorem \ref{thm:multi_conn}. Note that if $t \ll \tau_{c,n}$, then Lemma \ref{lem:iso_vertices} guarantees that there are isolated vertices with high probability, so that $B^*_t$ is disconnected. Hence we need to prove that when $t \gg \tau_{c,n}$, $B^*_t$ is connected with high probability. Towards this, we need the following structural result. Pittel \cite{pittel2010random} makes a similar observation in the unipartite case.
\begin{proposition}[Multigraph structure beyond connectedness threshold]\label{prop:multi_struc} Assume $\liminf\limits_{n \to \infty}\frac{t}{\tau_{c,n}} \in (0,\infty]$. Then \textbf{whp}, $B^*_t$ has the following structure. There is one big connected component $\sC_t$, and if this component does not cover all the vertices, then the rest of the components are simply isolated vertices.
\end{proposition}
To prove this proposition, we need a few technical results.
\begin{lemma}\label{lem:tech_lem_ratio_integral_UB}
    Let $M$ be a non-negative integer and $c,N$ are non-negative reals. Then $\prod_{i=0}^{M}\left(1-\frac{c}{N+i} \right)\leq \left(\frac{N}{N+M} \right)^c$, and $\prod_{i=0}^M\left(1+\frac{c}{N+i}\right) \leq \left(\frac{M+N}{N}\right)^c$.
\end{lemma}
\begin{proof}
   Using the bounds $1-x \leq e^{-x}$ and $\sum_{i=0}^M\frac{1}{N+i}\leq \int_0^M \frac{dy}{N+y}$, \begin{align*}
        \prod_{i=0}^{M}\left(1-\frac{c}{N+i} \right)\leq \exp{\left(-c\int_{0}^{M}\frac{dy}{N+y} \right)}=\left(\frac{N}{N+M} \right)^{c}. 
    \end{align*}
    The argument for the second assertion is identical, now using $1+x\leq e^x$.
\end{proof}
Next, we state a couple of results assuming which we first finish the proof of Proposition \ref{prop:multi_struc} and Theorem \ref{thm:multi_conn}. For a sequence $x_n$, we say it goes to zero \emph{superpolynomially fast}, if $(L_n+R_n)^c x_n=o(1)$ for any $c>0$. We denote this by $x_n=o_{\mathrm{sp}}(1)$
\begin{lemma}\label{lem:tech_lem_H}
Let $t_1=t_1(n)$ be a sequence of $n$ satisfying $t_1\gg (L_n+R_n)$. For any $k,j$ integers with $1\leq k \leq L_n$ and $1\leq j \leq R_n$, and $t=t_n$ satisfying $\liminf\limits_{n \to \infty}\frac{t}{\tau_{c,n}}\in (0,\infty]$, 
\begin{align*}
    \binom{t}{t_1}\cdot \frac{(k \alpha)_{t_1}((L_n-k)\alpha)_{t-t_1}}{(L_n \alpha)_{t}}\cdot \frac{(j \beta)_{t_1}((R_n-j)\beta)_{t-t_1}}{(R_n \beta)_{t}}=o_{\mathrm{sp}}(1).
\end{align*}
\end{lemma}
Let us not provide the proof of this result, a proof can be worked out essentially following the argument of \cite[Proof of Theorem 2, Step 1]{pittel2010random}.

\begin{lemma}\label{lem:tech_lem_ratio}
     Let $m=m_n$ be any sequence satisfying $m\ll \tau_{c,n}$. Then
\begin{align*}
\sum_{k=1}^{L_n}\sum_{j=1}^{R_n}\sum_{t_1=k\vee j}^{m}\binom{L_n}{k}\cdot \binom{R_n}{j}\cdot \binom{t}{t_1}\cdot \frac{(k \alpha)_{t_1}((L_n-k)\alpha)_{t-t_1}}{(L_n\alpha)_t}\cdot \frac{(j \beta)_{t_1}((R_n-j)\beta)_{t-t_1}}{(R_n\beta)_t}=o(1).
\end{align*}
\end{lemma}
We can now prove Proposition \ref{prop:multi_struc}.
\begin{proof}[Proof of Proposition \ref{prop:multi_struc}]
First consider the case $\alpha\neq \beta$. To conclude the proposition, it is enough to show that there is a component comprising of \emph{all} the vertices from $\rho^{-1}(\alpha \wedge \beta)$. Because of the bipartite structure, any remaining component must be an isolated vertex. We note that using Lemma \ref{lem:iso_vertices}, since $t\geq (L_n+R_n)^{1+(\alpha\wedge \beta)^{-1}} \gg (L_n+R_n)^{1+(\alpha \vee \beta)^{-1}}$, $\cI(\rho^{-1}(\alpha \vee \beta))=0$ \textbf{whp}. Hence, we need to rule out that there is a component of $B^*_t$ with $t_1 \in [1, t/2]$ many edges, with at most $|\rho^{-1}(\alpha \vee \beta)|-1$ many vertices from $\rho^{-1}(\alpha \vee \beta)$. Let us further consider the case $\alpha>\beta$, the argument for the other possibility is identical.

When $\alpha>\beta$, if there is a component of $B^*_t$ with $k \in [1,L_n-1]$ many vertices from $\cL$ with $t_1 \in [1,t/2]$ many edges, recalling \eqref{eq:def_B_t_event}, we note that $B^*_t\in \cB_t(A,B,t_1,0)$ for some $A \subset \cL$ with size $k$, $B \subset \cR$ is a proper subset of $\cR$ (meaning not either $\phi$ or $\cR$) and $k\vee |B| \leq t_1< t/2$. Using Proposition \ref{prop:edge_part_form} with $y=0$, the probability of the event of interest is therefore at most
\begin{align*}
    \sum_{k=1}^{L_n-1}\sum_{j=1}^{R_n-1}\sum_{t_1=k\vee j}^{t/2}\binom{L_n}{k}\cdot\binom{R_n}{j}\cdot {\binom{t}{t_1}}\cdot \frac{(k\alpha)_{t_1}((L_n-k)\alpha)_{t-t_1}}{(\alpha L_n)_{t}}\cdot \frac{(|B|\beta)_{t_1}((R_n-|B|)\beta)_{t-t_1}}{(\beta R_n)_{t}}. \\\numberthis \label{eq:multi_struc_complement_prob}
\end{align*}
Consider any term in the above sum corresponding to $t_1\geq m_n:=(L_n+R_n)^{1+\eps}\gg L_n+R_n$ for some $0<\eps<(\alpha \wedge \beta)^{-1}$. Thanks to Lemma \ref{lem:tech_lem_H}, as there are at most polynomial in $(L_n+R_n)$ many terms in the sum, the total contribution coming from these terms in the summand is $o(1)$. The remaining sum,
\begin{align*}
    \sum_{k=1}^{L_n-1}\sum_{j=1}^{R_n-1}\sum_{t_1=k\vee j}^{m_n}\binom{L_n}{k}\cdot\binom{R_n}{j}\cdot {\binom{t}{t_1}}\cdot \frac{(k\alpha)_{t_1}((L_n-k)\alpha)_{t-t_1}}{(\alpha L_n)_{t}}\cdot \frac{(|B|\beta)_{t_1}((R_n-|B|)\beta)_{t-t_1}}{(\beta R_n)_{t}}=o(1)
\end{align*}
applying Lemma \ref{lem:tech_lem_ratio} with $m=(L_n+R_n)^{1+\eps}\ll \tau_{c,n}$. 
Finally, consider the case $\alpha=\beta$. In this case, by Lemma \ref{lem:iso_vertices}, it is possible to have isolated vertices from both $\cL$ and $\cR$. However, a similar argument using Lemmas \ref{lem:tech_lem_H} and \ref{lem:tech_lem_ratio}, one can show that the term \eqref{eq:multi_struc_complement_prob} for $\alpha=\beta$ is $o(1)$. \end{proof}
\begin{remark}
    Our argument shows that when $\alpha \neq \beta$, if $(L_n+R_n)^{1+(\alpha \vee \beta)^{-1}}\ll t \ll (L_n+R_n)^{1+(\alpha \wedge \beta)^{-1}}$, the multigraph $B^*_t$ has the following structure \textbf{whp}: there is big connected component $\sC_t$ comprising of all vertices from $\rho^{-1}(\alpha \vee \beta)$, and the remaining components are all isolated vertices from $\rho^{-1}(\alpha \wedge \beta)$. 
\end{remark}

Having Proposition \ref{prop:multi_struc} in hand, the proof of Theorem \ref{thm:multi_conn} is straightforward. 
\begin{proof}[Proof of Theorem \ref{thm:multi_conn} for $\alpha \neq \beta$]
    Using Proposition \ref{prop:multi_struc}, the only way for $\cG^*_t$ to be disconnected is to have isolated vertices. The result now follows using the mutual independence of the processes $(D_u(\cdot))_{u \in V}$ and using Lemma \ref{lem:iso_vertices}.
\end{proof}
\begin{remark}
    Thus, the nature of connectivity for the $(\alpha,\beta)$-multigraph dynamics is similar to the E-R multigraph dynamics, in that both are determined by isolated vertices. This is also true in the unipartite case, as was shown by Pittel \cite[Theorem 2]{pittel2010random}.
\end{remark}
Finally, let us prove Lemma \ref{lem:tech_lem_ratio}. This is an extension of \cite[Proof of Theorem 2, Step 2]{pittel2010random}. Let us remark that we were hoping to omit out this proof as well, but the argument, even though in spirit similar to the mentioned reference, requires more subtle analysis, which is why we include a proof for completeness. 
\begin{proof}
Denote the sum in question by $\sum\limits_{k=1}^{L_n}\sum\limits_{j=1}^{R_n}\sum\limits_{t_1=k\vee j}^{m}\mathcal{S}(k,j,t_1)$. We split
\begin{align*}
    \sum_{1\leq k\leq L_n}\sum_{1\leq j \leq R_n, \atop j\geq k}\sum_{j\leq t_1 \leq m}\mathcal{S}(k,j,t_1)+\sum_{1\leq k\leq L_n}\sum_{1\leq j \leq L_n,\atop j< k}\sum_{k\leq t_1 \leq m}\mathcal{S}(k,j,t_1),\numberthis \label{eq:second_lem_split}
\end{align*}
with the convention that the sum over an empty set equals zero. Let us prove that the first term in \eqref{eq:second_lem_split} is $o(1)$, the analysis of the second term being similar, we omit it. Note that
\begin{align*}
    \frac{\mathcal{S}(k,j,t_1+1)}{\mathcal{S}(k,j,t_1)}=\frac{t-t_1}{t_1+1}\cdot \frac{(k\alpha+t_1)(j \beta+t_1)}{((L-k)\alpha+t-t_1-1)((R-j)\beta+t-t_1-1)}<\delta,
\end{align*}
for some $0< \delta<1$, whenever $t_1\leq m_n\ll t$,
so that the first sum in \eqref{eq:second_lem_split} equals up to a constant factor to 
\begin{align*}
    &\sum_{1\leq k \leq L_n}\sum_{1\leq j \leq R_n,\atop j \geq k}\binom{L_n}{k}\cdot \binom{R_n}{j}\cdot \binom{t}{j}\cdot \frac{(k \alpha)_{j}((L_n-k)\alpha)_{t-j}}{(L_n\alpha)_t}\frac{(k \beta)_{j}((R_n-j)\beta)_{t-j}}{(L_n\beta)_t}
    \\&=\sum_{1\leq k \leq L_n}\sum_{1\leq j \leq R_n,\atop j \geq k}\mathcal{S}_1(k,j),\;\;\text{say.}
\end{align*} 
Next, consider the ratio
\begin{align*}
    &\frac{\cS_1(k,j+1)}{\cS_1(k,j)}\\&=\frac{(R_n-j)(t-j)}{(j+1)^2}\frac{(k\alpha+j)}{((L_n-k)\alpha+t-j-1)}\frac{(k\beta+j)}{((R_n-j)\beta+t-j-1)}\prod_{p=0}^{t-j-2}\left(\frac{(R_n-j-1)\beta+p}{(R_n-j)\beta+p}\right).
\end{align*}
Bounding the product term from above by $1$, we obtain $\frac{\cS_1(k,j+1)}{\cS_1(k,j)}=O(\frac{R_n}{t})\ll 1$,
so that the sum 
\begin{align*}
    &\sum_{1\leq k \leq L_n}\sum_{1\leq j \leq R_n,j \geq k}\mathcal{S}_1(k,j)\\&=\left(\sum_{1\leq k \leq L_n}\binom{L_n}{k}\binom{R_n}{k}\binom{t}{k}\frac{(k \alpha)_{k}((L_n-k)\alpha)_{t-k}}{(L_n\alpha)_t}\frac{(k \beta)_{k}((R_n-k)\beta)_{t-k}}{(R_n\beta)_t} \right)\cdot(1+O(R_n/t))^{-1}\\&=\left(\sum_{1\leq k \leq L_n}\cS_2(k) \right)\cdot(1+O(R_n/t))^{-1},\;\;\text{say.}
\end{align*}
Finally, consider the ratio
\begin{align*}
    \frac{\cS_2(k+1)}{\cS_2(k)}=\frac{L-k}{k+1}\frac{R-k}{k+1}\frac{t-k}{k+1}\prod_{\square \in \{L,R\}}\psi(\square),
\end{align*}
where for $\square \in \{\cL,\cR\}$ fixed,
\begin{align*}
    \psi(\square)=\frac{(k+1)\rho(\square)+k}{(|\square|-k)\rho(\square)+t-k-1}\cdot \prod_{i=0}^{k-1}\left(1+\frac{\rho(\square)}{k\cdot \rho(\square)+j} \right)\cdot \prod_{j=0}^{t-k-2}\left(1-\frac{\rho(\square)}{(|\square|-k)\rho(\square)+j} \right).
\end{align*}
Using Lemma \ref{lem:tech_lem_ratio_integral_UB} for the product terms in the last display, we note that
\begin{align*}
    \psi(\square)\leq \frac{(k+1)\rho(\square)+k}{(|\square|-k)\rho(\square)+t-k-1} \cdot \left(\frac{k-1+k\cdot \rho(\square)}{k\cdot \rho(\square)} \right)^{\rho(\square)}\cdot\left(\frac{(|\square|-k)\rho(\square)}{(|\square|-k)\rho(\square)+t-k-2} \right)^{\rho(\square)}.
\end{align*}
Overall, for any $\square \in \{\cL,\cR\}$, $\psi(\square)=O\left(\frac{k |\square|^{\rho(\square)}}{t^{1+\rho(\square)}} \right).$ Hence, recalling $L_n=\Theta(R_n)$, we note that the product term $$\prod_{\square \in \{\cL,\cR\}}\psi(\square)=O\left(\frac{k^2 L_n^{\alpha+\beta}}{t^{2+\alpha+\beta}} \right).$$
As a consequence, 
\begin{align*}
    \frac{\cS_2(k+1)}{\cS_2(k)}=O\left(\frac{L_n^{2+\alpha+\beta}}{t^{1+\alpha+\beta}} \right)=o(1),
\end{align*}
since $t=\Omega\left(L_n^{1+(\alpha \wedge \beta)^{-1}}\right).$ Therefore, we obtain
\begin{align*}
    \sum_{k=1}^{L_n} \sum_{j=1}^{R_n} \sum_{t_1=k \vee j}^{t}\cS(k,j,t_1)\leq 2 \cS_2(1)=2(\alpha \beta L_n R_n t)\prod_{\square \in \{\cL,\cR\}}\left(\frac{((|\square|-1)\cdot\rho(\square))_{t-1}}{(|\square| \cdot \rho(\square))_t} \right).
\end{align*}
For fixed $\square\in \{\cL,\cR\}$, we note that
\begin{align*}
    \frac{((|\square|-1)\cdot\rho(\square))_{t-1}}{(|\square| \cdot \rho(\square))_t}=\frac{1}{|\square|\cdot \rho(\square)+t-1}\cdot \prod_{j=0}^{t-2}\left(1-\frac{\rho(\square)}{|\square| \cdot \rho(\square)+j} \right)=O\left(\frac{|\square|^{\rho(\square)}}{t^{1+\rho(\square)}}\right),
\end{align*}
where to obtain the the last equality above we again use Lemma \ref{lem:tech_lem_ratio_integral_UB}. Hence,
\begin{align*}
    2\cS_2(1)=O\left(\frac{L_n^{2+\alpha+\beta}}{t^{1+\alpha+\beta}} \right)=o(1),
\end{align*}
which finishes the proof.
\end{proof}

\subsection{Simple graph connectedness lower bound}
In this section we prove Theorem \ref{thm:sg_LB}. The idea is to use the measure change lemma, Lemma \ref{lem:measure_change}. Let $\cA$ be the event of connectedness, then to prove the simple graph is disconnected after the inclusion of $(L_n+R_n)^{1+\delta}$ many edges, we first estimate the prefactor in the RHS of \eqref{eq:mult_approx_bootstrap_UB}, and then try to choose $\delta>0$ carefully so that the postfactor $\Prob{B^*_t\;\;\text{is connected}}$ goes to zero fast enough to beat this prefactor. Here are the details:

\begin{proof}[Proof of Theorem \ref{thm:sg_LB}] Let  $(H_i)_{0 \leq i\leq t}$ be any multigraph sequence where $H_{i+1}$ is obtained from $H_i$ after the inclusion of exactly one edge. By \eqref{eq:exact_RN_form} and the monotonicity of $Q(H_i)$, 
\begin{align*}
    \Prob{(B_j)_{0\leq j \leq t}=(H_j)_{0\leq j \leq t}}\leq \left(1-\frac{4(Q(H_{t-1})+\gamma(t-1))}{(\alpha \beta L_nR_n)} \right)^{-(t-1)}\cdot \Prob{(B^*_j)_{0\leq j \leq t}=(H_j)_{0\leq j \leq t}}.
\end{align*}
Similar to \eqref{eq:mult_approx_bootstrap_UB}, for any event $\cA$ and $M>0$, we derive from the last display
\begin{align*}
    \Prob{(B_j)_{0\leq j \leq t}\in \cA}\leq \left(1-\frac{4M+\gamma(t-1))}{(\alpha \beta L_nR_n)} \right)^{-(t-1)} \Prob{(B^*_j)_{0\leq j \leq t}\in \cA}+\Prob{Q(B^*_{t-1})>M}.\\\numberthis \label{eq:sg_ctd_UB1}
\end{align*}
Let us first show that the second term on the RHS above is $o(1)$ for a right choice of $M$ that we also make along the way. Recalling the definition of $Q$ from \eqref{eq:def_Q}, we bound this term from above as
\begin{align*}
    \Prob{\sum_{u \in \cL}(d_u(t))^3>M/2}+\Prob{\sum_{v \in \cR}(d_v(t))^3>M/2}. \numberthis \label{eq:sg_term2}
\end{align*}
We employ the continuous time embedding of Section \ref{sec:cts_time}. Recall the continuous time processes $D_v(\cdot)$ corresponding to any vertex $v \in \cL \cup \cR$, for any $\square \in \{\cL,\cR\}$ the stopping times $\tau_{\square,i}$ for $i \geq 1$, and the fact that for any $u \in \square$, $d_u(t)$ has the same distribution as $D_u(\tau_{\square,t})$. Thus, using time monotonicity of the processes $D_u(\cdot)$, and using Lemma \ref{lem:stopping_time}, for $t=t(\delta):=(L_n+R_n)^{1+\delta}$ for any $\delta>0$, we note that up to an additive error of $o(1)$, \eqref{eq:sg_term2} is at most    
\begin{align*}
    \Prob{\sum_{u \in \cL}(D_u(\tau_n^{\cL,+}(t,s)))^3>M/2}+\Prob{\sum_{u \in \cR}(D_u(\tau_n^{\cR,+}(t,s))^3>M/2},\numberthis \label{eq:sg_term2_cont}
\end{align*}
for any $s$ satisfying $\frac{t}{\sqrt{L}}\ll s \leq t$, and where recall $\tau_n^{\square,\pm}(t,s)$ from \eqref{eq:def_ub_lb_stop_times}. From now on let us fix $s=(L_n+R_n)^{1/2+\delta+\delta_1}$ for a suitable $\delta_1>0$ satisfying $\delta+\delta_1<1/2$ which we will suitably choose later. Now, recall that for any $r>0$, the distribution of $D_u(r)$ is $\mathrm{NB}(\rho(\square),1-e^{-r})$ for $u \in \square$ for any $\square \in \{\cL,\cR\}$. Thus, using Lemma \ref{lem:nb_properties}, we note that $\Exp{D_u(\tau_n^{\square,+}(s,t))}=\Theta\left((L_n+R_n)^{\delta} \right)$. In particular, using Lemma \ref{lem:nb_properties} for $k=3$, and a Markov's inequality, we note that for $M=(L_n+R_n)^{3\delta+1+\eps}$ for arbitrary $\eps>0$, the term \eqref{eq:sg_term2_cont} is $o(1)$.

We now turn to the first term on the RHS of \eqref{eq:sg_ctd_UB1}. Observe that for the choices $M=(L_n+R_n)^{3\delta+1+\eps}$ and $t=(L_n+R_n)^{1+\delta}$, using the bound $1-x>\exp(-x^{1-\delta'})$ for any $\delta'>0$ and when $x$ is small, 
\begin{align*}
    \left(1-\frac{4M+\gamma(t-1))}{(\alpha \beta L_nR_n)} \right)^{-(t-1)}=\exp\left(O\left((L_n+R_n)^{4\delta+\eps} \right)\right).\numberthis \label{eq:RND_fac_UB}
\end{align*}

Now, consider any $\delta \in (0,\cZ(\alpha,\beta))$. Using \eqref{eq:sg_ctd_UB1} and \eqref{eq:RND_fac_UB}, we write
\begin{align*}
    &\Prob{B_{t(\delta)}\;\;\text{is connected}} \leq \exp\left(O\left((L_n+R_n)^{4\delta+\eps} \right)\right) \cdot \Prob{B^*_{t(\delta)}\;\;\text{has no isolated vertices}}+o(1).
\end{align*}
We now estimate the term $\Prob{B^*_{t(\delta)}\;\;\text{has no isolated vertices}}$ using the continuous processes $D_u(\cdot)$. Note that by independence,
\begin{align*}
    \Prob{B^*_{t(\delta)}\;\;\text{has no isolated vertices}}&=\Prob{D_u(\tau_{t(\delta)})>0}^{L_n} \cdot\Prob{D_v(\tau_{t(\delta)})>0}^{R_n}.\numberthis \label{eq:iso_bound}
\end{align*}
For any $\square \in \{\cL,\cR\}$ and $u \in \square$, using the monotonicity of $D_u(\cdot)$, we bound 
\begin{align*}
    \Prob{D_u(\tau_t)>0}\leq \Prob{D_u(\tau_n^{\square,+}(t,s))>0}+\Prob{\tau_t>\tau_n^{\square,+}(t,s)}, \numberthis \label{eq:iso_bound_L}
\end{align*}
and note that since $D_u(\tau_n^{\square,+}(t,s))\sim \mathrm{NB}\left(\rho(\square), \frac{|\square|^{1+\delta}+s}{\rho(\square) |\square|+|\square|^{1+\delta}+s}\right)$, with some easy analysis, the first term on the RHS of \eqref{eq:iso_bound_L} is at most $\exp{\left(-|\square|^{-\rho(\square) \delta (1+\eps_1)} \right)}$ for arbitrary $\eps_1>0$. Next, from Lemma \ref{lem:stopping_time}, note that $\Prob{\tau_t>\tau_n^{\square,+}(t,s)}$ is at most $C\frac{|\square|^{1+2\delta}}{s^2}=C|\square|^{-2\delta_1}$ for some constant $C>0$, where recall $s=(L_n+R_n)^{1/2+\delta+\delta_1}=\Theta\left(|\square|^{1/2+\delta+\delta_1}\right)$. Thus, choosing $\eps_1>0$ small enough so that $\delta<[2(1+\rho(\square)(1+\eps_1))]^{-1}$, made possible by the fact $\delta<[2(1+\rho(\square))]^{-1}$ (implied by $\delta<\cZ(\alpha,\beta)$, recall \eqref{eq:def_cZ}), and then choosing $\delta_1>\rho(\square) \delta (1+\eps_1)$, which satisfies $\delta+\delta_1<1/2$, we note that
\begin{align*}
    &\Prob{D_u(\tau_n^{\square,+}(t,s))>0}+\Prob{\tau_t>\tau_n^{\square,+}(t,s)}\\& \leq \exp{\left(-|\square|^{-\rho(\square) \delta (1+\eps_1)} \right)}+|\square|^{-2\delta_1}\leq \exp{\left(-(1-c)|\square|^{-\rho(\square)\delta(1+\eps_1)}\right)},
\end{align*}
for any constant $c>0$. As a consequence, from \eqref{eq:iso_bound} we note that for constants $c,c'>0$,
\begin{align*}
    \Prob{\cG^*_{t(\delta)}\;\;\text{has no isolated vertices}}&\leq \exp{\left(-(1-c)L^{1-\alpha \delta (1+\eps_1)}-(1-c')L^{1-\beta \delta (1+\eps_2)} \right)},
\end{align*}
Recalling \eqref{eq:RND_fac_UB}, we thus note that $\Prob{B_{t(\delta)}\;\;\text{is disconnected}}=o(1)$ if we can further choose $\eps,\eps_1,\eps_2>0$ sufficiently small such that
\begin{align*}
    \exp{\left(L^{4\delta+\eps}-(1-c)L^{1-\alpha \delta (1+\eps_1)}-(1-c')L^{1-\beta \delta (1+\eps_2)}\right)}=o(1),
\end{align*}
which is indeed possible when either $\delta<\frac{1}{4+\alpha}\wedge\frac{1}{\beta}$, or $\delta<\frac{1}{4+\beta}\wedge \frac{1}{\alpha}$, and this is implied by the fact that $\delta<\cZ(\alpha,\beta)$, recall \eqref{eq:def_cZ}. \end{proof}
\section{Discussion}\label{sec:disc}
\subsection{Connectivity threshold lower bound for the unipartite case}\label{sec:sg_uni_lb} For the unipartite case as was considered in \cite{pittel2010random} and \cite{janson2021preferential}, an argument similar to the proof of Theorem \ref{thm:sg_LB} gives that the corresponding simple graph process is disconnected \textbf{whp} after the inclusion of $n^{1+\delta}$ edges for any $\delta \in (0,1/(4+\alpha))$, here $n$ is the total number of vertices, and recall edges are included with probability proportional to $(d_u(t)+\alpha)(d_v(t)+\alpha)$. Thus, the simple graph connectivity threshold is at least $n^{1+1/(4+\alpha)}$ many edges. As mentioned in the introduction, this result is a small first step towards Pittel's conjecture that the connectivity threshold for the simple graph is $n^{1+\alpha^{-1}}$. 
\subsection{Multipartite case}
The natural multipartite case, where the vertex set $V$ is multipartitioned as $V=V_1\cup\dots \cup V_k$, and edges between non-adjacent vertices appear with probability proportional to $(d_u(t)+\alpha_i)(d_v(t)+\alpha_j)$ for $u \in V_i$ and $v \in V_j$, turns out to be more difficult. The analog of Proposition \ref{prop:BCM_approx} can be proved, where in this case one has to consider a multipartite configuration model (MCM). This is a generalization of the BCM. Here for each vertex its degree into each of the partitions is provided. Then one constructs BCM's between each partition pair, and the union of all of these edges form the MCM. For a brief discussion of this model, see van der Hofstad \cite[Chapter 9]{van2024random}. Proving an analog of Theorem \ref{thm:giant} following the recipe of this paper (which in turn is inspired by the route-map of the paper \cite{janson2021preferential}) requires -- firstly, establishing the giant component threshold for the multipartite configuration model when the degree sequences between partition pairs are given; secondly, and more importantly, for a given partition pair, to understand the behavior of these degrees between the said pair. The latter seems the harder step as these degrees are not negative binomials anymore (unless all the $\alpha_i$ are equal), however, the former is an interesting question in its own right. We leave these explorations for future work.

\medskip 
%\paragraph{Competing interests.} The author(s) declare none.
%\paragraph{Funding statement.} This work received no specific grant from any funding agency, commercial or not-for-profit sectors.

\medskip

\bibliographystyle{abbrvnat}
\bibliography{ref}
\end{document}